%% file: weak_single_layer_6.tex
\documentclass[10pt]{article}
\usepackage{amsmath, amsfonts, amssymb,  mathrsfs, color, amsthm, epsfig, color, graphics, mathtools,oubraces, enumerate,authblk,esint,empheq}
  \usepackage[all]{xy}
 \usepackage[latin1]{inputenc}
\newtheorem{theorem}{Theorem}[section]
\newtheorem*{theorem*}{Theorem}
\newtheorem{lemma}{Lemma}[section]
\newtheorem{prop}{Proposition}[section]
\newtheorem{definition}{Definition}[section]
\newtheorem{cor}{Corollary}[section]
\usepackage{anysize} 
\marginsize{2cm}{2cm}{2cm}{2cm}

\newcommand{\llangle}{\langle\mspace{-4mu}\langle}

\newcommand{\rrangle}{\rangle\mspace{-4mu}\rangle}

\DeclareMathAccent{\mpetito}{\mathalpha}{operators}{23}

\newcommand{\dy}{\,{\rm d}y}
\newcommand{\forallt}{\qquad\text{for all }}

\newcommand{\fct}[4]{\arraycolsep=1.4pt\begin{array}{rcl}{#1}&\longrightarrow&{#2}\\{#3}&\longmapsto&{#4}\end{array}}

\providecommand{\keywords}[1]
{
  \small	
  \textbf{\textit{Keywords---}} #1
}
\providecommand{\MSC}[1]
{
  \small	
  \textbf{\textit{MSC---}} #1
}
\title{Square integrable surface potentials on non-smooth domains and application to the Laplace equation in $L^2$}
\begin{document}
%
\date{\today}
\author{Alexandre Munnier\\ {alexandre.munnier@univ-lorraine.fr}}
\affil{Universit\'e de Lorraine, CNRS, Inria, IECL, F-54000 Nancy, France}
\maketitle
\begin{abstract}Motivated by applications in fluid dynamics involving the harmonic Bergman projection,
we aim to extend the theory of single and double layer potentials (well documented for functions with $H^1_{\ell oc}$ regularity) 
to locally square 
integrable functions. Having in mind numerical simulations for which functions are usually defined on a polygonal mesh, we wish this theory to cover 
the cases of non-smooth domains (i.e.with Lipschitz continuous or polygonal boundaries).
\end{abstract}
\noindent\keywords{Layer potentials in $L^2$, non-smooth domains, Laplace equation in $L^2$.}

\noindent\MSC{35C15 , 35D30, 45A05  }
\section{Introduction}
Let $\varOmega$ be  a smooth bounded domain in the plan. The harmonic Bergman projection is the orthogonal projection in $L^2(\varOmega)$
onto the closed subspace of harmonic functions (see \cite[Chap. 8]{Axler:2001aa} and \cite{Straube:1986aa}). This operator is 
known mainly  for  playing an important role in complex analysis and operator theory but 
has also applications in the field of partial differential equations 
(\cite{Kracht:1988aa}, \cite[Chap. 4]{Krantz:2013aa}). In fluid dynamics, it appears in the article \cite{Plotnikov:2009aa} and more 
recently in \cite{Lequeurre:2020aa} for the analysis of the Navier-Stokes equations in non-primitive variables (stream function and vorticity). Indeed, for an incompressible fluid flow, the vorticity field is orthogonal in $L^2$ to the harmonic functions (see \cite{Lequeurre:2020aa} and 
references therein).
\par
The Bergman projection is a kernel operator but this kernel can be explicitly computed only for  particular geometries  (when $\varOmega$ is a disk or a half plane for instance). 
From a numerical point of view, the discretization of the Bergman projection requires the inversion of the mass matrix corresponding to the $L^2$ scalar product restricted to the subspace of harmonic functions. For this purpose, a discrete basis of harmonic functions in $L^2$ is needed and an efficient way to construct such a basis consists in using boundary elements and layer potentials. However, 
while the theory of layer potentials in $H^1_{\ell oc}(\mathbb R^2)$ is well documented (see the classical book \cite{Atkinson:1997aa} for instance), 
little is known on locally 
square integrable layer potentials. In this paper, we aim to  provide a theoretical framework for this notion.
Furthermore, in numerical simulations, functions are usually defined on a polygonal mesh, so we want to cover this case, which adds a substantial difficulty. 
\par
In its classical meaning, the single layer potential 
maps 
the Sobolev space $H^{-1/2}(\varGamma)$ into $H^1_{\ell oc}(\mathbb R^2)$ 
($\varGamma$ stands here for a Lipschitz continuous Jordan curve). A natural guess is that  the $H^1_{\ell oc}$
 regularity  could be lowered to $L^2_{\ell oc}$ by extending the single layer potential to the space $H^{-3/2}(\varGamma)$.
However the space $H^{3/2}(\varGamma)$, and then 
also its dual space $H^{-3/2}(\varGamma)$ are ill defined 
on a Lipschitz continuous boundary, any intrinsic definition of these spaces requiring that the boundary be at least of class $\mathcal C^{1,1}$. 
On the other hand, denoting by $\gamma_d$ 
the classical Dirichlet trace operator on $\varGamma$, the space $\mathcal H^{3/2}(\varGamma)=
\gamma_d H^2_{\ell oc}(\mathbb R^2)$, although complex to describe in terms of Sobolev regularity, is well defined (and coincides with $H^{3/2}(\varGamma)$ when $\varGamma$
 is smooth).
The main idea of the paper is to define the single-layer potentials as Laplacians of biharmonic functions in $\mathbb R^2\setminus\varGamma$, the asymptotic behavior of the functions being taken into account by introducing an appropriate functional framework based on weighted Sobolev spaces. This approach will prove successful  
and will allow to extend the single layer potential 
 to the space $\mathcal H^{-3/2}(\varGamma)$.
\par
Considering the double layer 
potential,  based on similar arguments, it will be extended to $\mathcal H^{-1/2}(\varGamma)$, the dual space of 
$\mathcal H^{1/2}(\varGamma)=\gamma_nH^2_{\ell oc}(\mathbb R^2)$, where $\gamma_n$ stands for the Neumann trace operator on 
$\varGamma$. It is worth noticing that $\mathcal H^{1/2}(\varGamma)$ is equal to $H^{1/2}(\varGamma)$ when $\varGamma$ is smooth but 
this is no longer true as soon as $\varGamma$ has corners for instance. 
\par
Throughout  the paper, we will assume without loss of generality that the logarithmic capacity of $\varGamma$ is lower than 1, using translation and dilatation of the coordinates system if necessary (see \cite[Page 263]{McLean:2000aa} on this matter).
Roughly speaking, we shall prove the following result
(that will be rigorously reformulated in Theorem~\ref{exten_layer} thereafter):
\begin{theorem} 
\label{theo:ext_slp}
Let $\varGamma$ be a Lipschitz Jordan curve. Then the single layer potential, considered as an operator defined on $H^{-1/2}(\varGamma)$ 
valued in $L^2_{\ell oc}(\mathbb R^2)$ 
extends by density to a bounded operator on 
$\mathcal H^{-3/2}(\varGamma)$. The double layer potential, seen as an operator from $H^{1/2}(\varGamma)$ into $L^2_{\ell oc}(\mathbb R^2)$ 
extends by density to a bounded operator on $\mathcal H^{-1/2}(\varGamma)$.
\end{theorem} 
Denote by $\varOmega^-$ the planar open set enclosed by $\varGamma$ and by 
$\varOmega^+$ its complement in $\mathbb R^2$. Providing that $\varGamma$ is a polygon,    we will be able to reach our initial goal  (to represent harmonic functions in $L^2_{\ell oc}$ by surface potentials) by proving (this result is rigorously reformulated later in Corollary~\ref{nkmpq:cor} and Corollary~\ref{kkqpolop:cor}):

\begin{theorem}
\label{gvbhnkopl}
Any harmonic function in $L^2(\varOmega^-)$ can be represented by the restriction 
to $\varOmega^-$ of a single or a double layer potential as defined in Theorem~\ref{theo:ext_slp}.
The same conclusion applies for harmonic functions in $L^2_{\ell oc}(\overline{\varOmega^+})$, assuming  additional properties on their asymptotic behaviors.
\end{theorem}
The remainder of the introduction is devoted to giving the reader an overview of the main steps of the analysis. As with Theorems~\ref{theo:ext_slp} and \ref{gvbhnkopl}, we do not seek to be rigorous 
at this stage but simply to give a taste of the results. 
For the sake of brevity, we will focus only on the single layer potential. %
\par
The first step of the analysis is to extend the notions of Dirichlet and Neumann traces to functions in $L^2_{\ell oc}(\mathbb R^2)$, 
harmonic in $\mathbb R^2\setminus\varGamma$.
This task will be carried out in the case where $\varGamma$ is a   curvilinear $\mathcal C^{1,1}$ polygon (i.e. a generalization of the notion 
of polygon for which the edges are $\mathcal C^{1,1}$ curves) and requires the introduction of the spaces:
$${\mathcal H}^{3/2}_n(\varGamma)=\big\{\gamma_du\,:\,u\in H^2_{\ell oc}(\mathbb R^2),\,\gamma_nu=0\big\}\qquad\text{and}
\qquad {\mathcal H}^{1/2}_d(\varGamma)=\big\{\gamma_nu\,:\,u\in H^2_{\ell oc}(\mathbb R^2),\,\gamma_du=0\big\}.$$
When $\varGamma$ is smooth, we simply have
${\mathcal H}^{3/2}_n(\varGamma)={\mathcal H}^{3/2}(\varGamma)=H^{3/2}(\varGamma)$ and 
${\mathcal H}^{1/2}_d(\varGamma)={\mathcal H}^{1/2}(\varGamma)=H^{1/2}(\varGamma)$.
However, all these equalities turn out to be false when 
$\varGamma$ is   a $\mathcal C^{1,1}$ polygon (this is what makes the analysis tricky).  The topologies of which these spaces are provided (and which will be specified thereafter) entail 
 the continuity 
and the density of the following inclusions:
$$\mathcal H_d^{1/2}(\varGamma)\subset \mathcal H^{1/2}(\varGamma)\subset L^2(\varGamma)\qquad\text{and}
\qquad \mathcal H^{3/2}_n(\varGamma)\subset \mathcal H^{3/2}(\varGamma)\subset H^{1/2}(\varGamma)\subset L^2(\varGamma).$$
As usual, we denote by ${\mathcal H}^{-3/2}_n(\varGamma)$ 
the dual space of ${\mathcal H}^{3/2}_n(\varGamma)$ and by ${\mathcal H}^{-1/2}_d(\varGamma)$ the dual space of ${\mathcal H}^{1/2}_d(\varGamma)$, using $L^2(\varGamma)$ as pivot space. More interesting for our purpose, the inclusions between dual spaces are also continuous and dense:  
$$L^2(\varGamma)\subset \mathcal H^{-1/2}(\varGamma)\subset \mathcal H^{-1/2}_d(\varGamma)\qquad\text{and}\qquad
L^2(\varGamma)\subset H^{-1/2}(\varGamma)\subset \mathcal H^{-3/2}(\varGamma)\subset \mathcal H^{-3/2}_n(\varGamma).$$ 
\begin{theorem}\label{theo:traces:intro} Any function $u$ in $L^2_{\ell oc}(\mathbb R^2)$, harmonic on both sides of $\varGamma$ admits
one-sided  
Dirichlet traces (denoted by $\gamma^-_du$ and $\gamma_d^+u$) in ${\mathcal H}^{-1/2}_d(\varGamma)$.
The function $u$ admits also one-sided Neumann traces 
(denoted by $\gamma^-_nu$ and $\gamma_n^+u$) in the space ${\mathcal H}^{-3/2}_n(\varGamma)$. 
Moreover, the 
trace operators $\gamma_d^\pm$ and $\gamma_n^\pm$ are the extensions by density of the classical trace operators defined for functions 
in $H^1(\varOmega^-)$ and in $H^1_{\ell oc}(\overline{\varOmega^+})$.
\end{theorem}
This notion of trace being clarified, we will turn again to the layer potentials and investigate the question of their Dirichlet and Neumann traces on $\varGamma$. 
Let $\mathscr S_\varGamma:H^{-1/2}(\varGamma)\longrightarrow H^1_{\ell oc}(\mathbb R^2)$ be the classical  single layer potential 
and recall the properties:
\begin{subequations}
\label{gftgvcdr}
\begin{equation}
\gamma_n^{\pm}\circ \mathscr S_\varGamma : H^{-1/2}(\varGamma)\longrightarrow H^{-1/2}(\varGamma)\qquad\text{and}\qquad 
\gamma_n^+\circ \mathscr S_\varGamma+\gamma_n^-\circ \mathscr S_\varGamma={\rm Id},
\end{equation}
this latter identity being usually called the ``jump relation''. Let now $\mathscr S_\varGamma^\dagger:\mathcal H^{-1/2}(\varGamma)
\longrightarrow L^2_{\ell oc}(\mathbb R^2)$ stands for the extended single layer potential defined in Theorem~\ref{theo:ext_slp}. 
According to Theorem~\ref{theo:traces:intro}
we have in this case:
\begin{equation}
\gamma_n^{\pm}\circ \mathscr S^\dagger_\varGamma : \mathcal H^{-3/2}(\varGamma)\longrightarrow \mathcal H_n^{-3/2}(\varGamma),
\end{equation}
\end{subequations}
where $\mathcal H^{-3/2}(\varGamma)$ is continuously and densely embedded in $\mathcal H_n^{-3/2}(\varGamma)$ but in general 
different from  $\mathcal H_n^{-3/2}(\varGamma)$, which suggests that the jump relation is not likely to apply in this case.
Surprisingly enough, the relation  is well and truly satisfied. More generally, concerning the traces of the single   layer 
potential, we will establish:
\begin{theorem}
\label{bnjkoplk}
The two one-sided  Dirichlet  traces  on   $\varGamma$ of a single layer potential (as defined in Theorem~\ref{theo:ext_slp}) coincide.
The ``jump'' across $\varGamma$ of the one-sided Neumann traces of a single layer potential of density $q\in \mathcal H^{-3/2}(\varGamma)$ is equal 
to $q$.
\end{theorem}
Actually, we will show that there do exist single layer potentials for which the one-sided Neumann traces are both in 
${\mathcal H}^{-3/2}_n(\varGamma)$ but not in ${\mathcal H}^{-3/2}(\varGamma)$, although their difference is in this latter space. 
This means that some singular contributions of the  normal derivatives cancel out by forming their difference. 
This notable   phenomenon seems to be 
typical of the single layer potential on non-smooth boundaries.
\par
The next point we shall discuss in the paper  is the solvability of the Dirichlet and Neumann Laplace equations with boundary data 
in $\mathcal H_d^{-1/2}(\varGamma)$ and $\mathcal H_n^{-3/2}(\varGamma)$. As for the 
Laplace equation with Dirichlet boundary conditions for example, we will prove:
%
%
\begin{theorem}
\label{vftploik}
Assume that $\varGamma$ is a (straight) polygon.
For every $p\in\mathcal H^{-1/2}_d(\varGamma)$ there exists a function $u^-\in L^2(\varOmega^-)$ 
harmonic in $\varOmega^-$ such that $\gamma_d^-u^-=p$ 
and there exists a function $u^+\in L^2_{\ell oc}(\overline{\varOmega^+})$ (with a suitable asymptotic behavior), harmonic in $\varOmega^+$ such that 
$\gamma_d^+u^+=p$. There is no uniqueness in general.
\end{theorem}
The existence of (non zero) harmonic functions in $L^2$ with vanishing Dirichlet data in a domain with corners has long been known (see for instance \cite{Dauge:1987tg} where an example of such a function is provided).

At this point, a kind of reciprocal of Theorem~\ref{bnjkoplk} will still be needed to prove Theorem~\ref{gvbhnkopl}. This result 
can be stated as follows:
\begin{theorem}
\label{vbghuijko}
Let $u$ be in $L^2_{\ell oc}(\mathbb R^2)$, harmonic in $\mathbb R^2\setminus\varGamma$, with an appropriate asymptotic 
behavior. If $u$ satisfies  
 $\gamma_d^-u=\gamma_d^+u$ then $q=\gamma_n^-u+\gamma_n^+u$ is in $\mathcal H^{-3/2}(\varGamma)$ and  $u=\mathscr S_\varGamma^\dagger q$. 
\par
\end{theorem}
The proof of Theorem~\ref{gvbhnkopl} now relies on a proper combination of Theorems~\ref{theo:traces:intro}, \ref{bnjkoplk} and \ref{vftploik}. 
Thus, denote by $p$ the one-sided trace $\gamma_d^-u^-$ of a given function $u^-$, harmonic in $L^2(\varOmega^-)$.
Theorem~\ref{vftploik} ensures the existence 
of a harmonic function  $u^+$ in $L^2_{\ell oc}(\overline{\varOmega^+})$ (with an appropriate asymptotic behavior) 
such that $\gamma_d^+u^+=p$. Define $q=\gamma_n^-u^-+\gamma_n^+u^+$ (the jump of the Neumann trace) and conclude, 
applying Theorem~\ref{vbghuijko} that 
$\mathscr S^\dagger_\varGamma q|_{\varOmega^\pm}=u^\pm$.%
\par
We shall   also provide a negative result, contrasting with what happens for harmonic functions $u$ such that $u|_{\varOmega^-}\in H^1(\varOmega^-)$
and $u|_{\varOmega^+}\in H^1_{\ell oc}(\overline{\varOmega^+})$. Indeed, such a function
harmonic in $\mathbb R^2\setminus\varGamma$ (and  with a suitable asymptotic behavior) can be represented as the sum 
of a single and double layer potentials. On the contrary:
\begin{theorem}
\label{theo_repres}
There exist functions in $L^2_{\ell oc}(\mathbb R^2)$, harmonic in $\mathbb R^2\setminus\varGamma$ (with a suitable asymptotic behavior) 
that cannot be represented as the sum of a single and a double layer potentials. 
\end{theorem}
%
We will end the article by studying the invertibility of the boundary operators introduced in \eqref{gftgvcdr}. Recall that the logarithmic capacity of $\varGamma$ is assumed to be lower than 1 and define:
$$\widetilde{\mathcal H}_n^{-3/2}(\varGamma)=\big\{q\in \mathcal H_n^{-3/2}(\varGamma)\,:\, \llangle q,\mathbf 1_\varGamma\rrangle_{-\frac32,\frac32,n}=0\big\},$$
where  $\llangle \cdot,\cdot\rrangle_{-\frac32,\frac32,n}$ stands   for the duality pairing 
on $\mathcal H^{-3/2}_n(\varGamma)\times \mathcal H^{3/2}_n(\varGamma)$ that extends the $L^2$ inner product.
\begin{theorem}
\label{theoo:6}
The bounded operators $\gamma_d\circ\mathscr S^\dagger_\varGamma:\mathcal H^{-3/2}(\varGamma)\longrightarrow {\mathcal H}^{-1/2}_d(\varGamma)$, 
$\gamma_n^-\circ\mathscr S^\dagger_\varGamma:\mathcal H^{-3/2}(\varGamma)\longrightarrow \widetilde{\mathcal H}^{-3/2}_n(\varGamma)$ 
and $\gamma_n^+\circ\mathscr S^\dagger_\varGamma:\mathcal H^{-3/2}(\varGamma)\longrightarrow {\mathcal H}^{-3/2}_n(\varGamma)$
are surjective but  not injective in general.
\end{theorem}
%
%
%
The paper is organized as follows: The following section is dedicated to the reminder of some basic notions about trace operators and   surface potentials. The main function spaces on which the analysis is based when $\varGamma$ is Lipschitz continuous, are introduced in Section~\ref{main_spaces}. They are used in Section~\ref{layer_pot} to extend the notion of surface potential to square integrable 
functions. From Section~\ref{sec:polyn} the boundary $\varGamma$ is assumed to be a $\mathcal C^{1,1}$ 
(curvilenar) polygon. This additional regularity allows the introduction of new function spaces involved in new trace theorems stated in the next section. The ``jump 
relations'' for surface potentials are proved in Section~\ref{SEC:jump}. From Section~\ref{sec:repres}, the analysis focuses on the case 
where $\varGamma$ is a straight polygon. Section~\ref{sec:repres}  is dedicated to  solvability issues for the 
Laplace equation with Dirichlet and Neumann boundary data. 
Finally, in section~\ref{SEC:trans}, we discuss some transmission problems and address the issue of
representing locally square-integrable harmonic functions as surface potentials. We end the paper with the proof of Theorem~\ref{theoo:6}.
\par
 For the ease of the reader, the appendix contains a list of the main function spaces and operators.
\section{Notations and recalls}
\subsection*{Geometric settings}

Let $\varOmega^-$ be an open and bounded planar domain whose boundary  $\varGamma$ is a Jordan curve. The (unbounded) complement 
of $\overline{\varOmega^-}$ is denoted by $\varOmega^+$. 
In the sequel, we shall consider four levels of regularity for 
$\varGamma$: 
It will be either of class 
$\mathcal C^{1,1}$ (referred to as the smooth case), either Lipschitz continuous (see \cite[Definition 1.2.1.1]{Grisvard:1985aa} 
for a precise definition of this notion), either a 
 $\mathcal C^{1,1}$ polygon (see \cite[Definition 1.4.5.1]{Grisvard:1985aa}), or simply a classical (straight) polygon.
In either case, the unit tangent vector field $\tau$ (oriented counterclockwise)  is a.e. well defined on $\varGamma$
and the same applies to the outer unit normal vector field  $n^-=-\tau^\perp$ and to the inner normal vector field $n^+=\tau^\perp$ 
(the superscript $\perp$ meaning that the vector is counterclockwise rotated of an angle $\pi/2$). To lighten the notations, we shall sometimes write simply $n$ instead of $n^-$. 
%
\subsection*{Traces on the boundary of a Lipschitz domain}
In this subsection, we collect some definitions and properties 
about the Dirichlet and Neumann trace operators in the case where 
$\varGamma$ is Lipschitz continuous. On the space
$$\mathscr D_{\varOmega^\pm}(\mathbb R^2)=\big\{u|_{\varOmega^\pm}\,:\,u\in\mathscr D(\mathbb R^2)\big\},$$
the one-sided 
Dirichlet and Neumann trace operators  are classically defined by:
\begin{equation}
\label{def_trace}
\fct{\gamma_d^\pm:\mathscr D_{\varOmega^\pm}(\mathbb R^2)}{L^2(\varGamma)}{u}{u|_{\varGamma}}\qquad\text{and}\qquad
\fct{\gamma_n^\pm:\mathscr D_{\varOmega^\pm}(\mathbb R^2)}{L^2(\varGamma)}{u}{\nabla u\cdot n^\pm|_\varGamma.}
\end{equation}
According to \cite[\S 9.2]{Agranovich:2015us}, when $\varGamma$ is   Lipschitz continuous, the sobolev space $H^s(\varGamma)$ is well (invariantly) defined only for those indices 
$s$ that belong to $[-1,1]$ and rephrasing \cite[Theorem 9.2.1]{Agranovich:2015us} (or 
\cite[Theorem 3.38]{McLean:2000aa}), we have:
\begin{theorem}
\label{dirichl_bound}
The one-sided Dirichlet trace operators $\gamma_d^\pm$ extend by density to bounded  operators from $H^{s+1/2}(\varOmega^\pm)$
to $H^{s}(\varGamma)$ for every  $0<s<1$.
\end{theorem}
%
According to \cite[Theorem 1]{Marschall:1987wa} we can also state:
%
\begin{theorem}
\label{kernel}The Dirichlet and Neumann trace operators \eqref{def_trace} extend by density to bounded operators on $H^2(\varOmega^\pm)$ 
(valued in $L^2(\varGamma)$) and $\ker \gamma_d^\pm\cap\ker \gamma_n^\pm=H^2_0(\varOmega^\pm)$,
where we recall that $H^2_0(\varOmega^\pm)$ is the closure of ${\mathscr D(\varOmega^\pm)}$ in $H^2(\varOmega^\pm)$.
\end{theorem}
%
The Neumann trace operator can actually be defined on a larger space than $H^2(\varOmega^\pm)$, namely on:
$$H^1(\varOmega^\pm,\Delta)=\big\{u\in H^1(\varOmega^\pm)\,:\,\Delta u \in L^2(\varOmega^\pm)\big\}.$$
Thus, according to \cite[Lemma 4.3]{McLean:2000aa}, for every $u\in H^1(\varOmega^\pm,\Delta)$, there exists a unique $g_u\in H^{-1/2}(\varGamma)$ such 
that:
$$\big\langle g_u,\gamma_d^\pm v\big\rangle_{-\frac12,\frac12}=(\Delta u,v)_{L^2(\varOmega^\pm)}-(\nabla u,\nabla v)_{L^2(\varOmega^\pm;\mathbb R^2)}
\forallt v\in H^1(\varOmega^\pm),$$
where $\langle\cdot,\cdot\rangle_{-\frac12,\frac12}$ stands for the duality bracket between the spaces $H^{-1/2}(\varGamma)$ 
and $H^{1/2}(\varGamma)$, that extends the $L^2$ inner product. Since $\mathscr D_{\varOmega^\pm}(\mathbb R^2)$ is dense in 
$H^1(\varOmega^\pm,\Delta)$ (see \cite[Lemma 1.5.3.9]{Grisvard:1985aa}), we are allowed to denote $g_u=\gamma^\pm_n u$ and we have
(see  \cite[Theorem 4.4]{McLean:2000aa} for the Green's identity):
\begin{prop}
\label{prop:21}
The Neumann trace operators $\gamma^\pm_n$ defined in \eqref{def_trace} extend   by density to    bounded operators from $H^1(\varOmega^\pm,\Delta)$ 
into $H^{-1/2}(\varGamma)$. Moreover, the second Green's identity holds:
\begin{equation}
\label{green}
(\Delta u,v)_{L^2(\varOmega^\pm)}-(u,\Delta v)_{L^2(\varOmega^\pm)}=\big\langle \gamma^\pm_n u,\gamma^\pm_0v\big\rangle_{-\frac12,\frac12}-
\big\langle \gamma^\pm_n v,\gamma^\pm_0u\big\rangle_{-\frac12,\frac12}\forallt u,v\in  H^1(\varOmega^\pm,\Delta).
\end{equation}
\end{prop}
The space $H^{3/2}(\varOmega^\pm,\Delta)=\big\{u\in H^{3/2}(\varOmega^\pm)\,:\,\Delta u \in L^2(\varOmega^\pm)\big\}$ (provided 
with the graph norm) is a subspace of $H^1(\varOmega^\pm,\Delta)$ and according to \cite[Lemma 3.2]{Gesztesy:2011aa}:
\begin{prop}
\label{prop:H32}
The operators $\gamma_n^\pm:H^{3/2}(\varOmega^\pm,\Delta)\longrightarrow L^2(\varGamma)$ are bounded and onto.
\end{prop}
Finally, the following density result will be useful in the sequel:
\begin{prop}
\label{prop:den}
The spaces  $\gamma_d \mathscr D_{\varOmega^\pm}(\mathbb R^2)$ and $\gamma_n \mathscr D_{\varOmega^\pm}(\mathbb R^2)$ 
are  dense in $L^2(\varGamma)$.
\end{prop}
The first assertion is proved in \cite[page 88]{Gesztesy:2011aa} and the second results from Proposition~\ref{prop:H32} 
and \cite[Lemma 3]{Costabel:1998aa}.
\par
\subsection*{Surface potentials on a Lipschitz boundary}
A general presentation of the theory of surface potentials on the boundary of a Lipschitz domain can be found in the book \cite{McLean:2000aa}, to which 
we will refer in the  following for more details on this subject. 
For the ease of the reader, let us recall some basics:
The fundamental solution of the Laplace's equation is defined by:
$$G(x)=-\frac{1}{2\pi}\ln|x|\forallt x\in\mathbb R^2\setminus\{0\}.$$
The single layer potential is the weakly singular integral operator  
defined for any $q\in L^2(\varGamma)$ by:
$$\mathscr S_\varGamma q(x)=\int_\varGamma G(x-y)q(y)\,{\rm d}y\forallt x\in\mathbb R^2\setminus \varGamma,$$
and extended by density to a bounded operator $\mathscr S_\varGamma:H^{-1/2}(\varGamma)\longrightarrow H^1_{\ell oc}(\mathbb R^2)$.
The double layer potential is the singular integral operator: $\mathscr D_\varGamma:H^{1/2}(\varGamma)\longrightarrow H^1_{\ell oc}(\mathbb R^2)$, defined by:
$$\mathscr D_\varGamma p(x)=\int_\varGamma \nabla G(x-y)\cdot n(y)q(y)\,{\rm d}y\forallt x\in\mathbb R^2\setminus\varGamma.$$
The single layer potentiel and the double layer potential  both admit one-sided Dirichlet 
and Neumann traces on both sides of $\varGamma$. In \cite{McLean:2000aa} it is proved that the following operators are well defined and bounded:
\begin{alignat*}{3}
\gamma_d^\pm\circ \mathscr S_\varGamma:H^{-1/2}(\varGamma)&\longrightarrow H^{1/2}(\varGamma)&\qquad\text{and}\qquad&
\gamma_d^\pm\circ \mathscr D_\varGamma:H^{1/2}(\varGamma)\longrightarrow H^{1/2}(\varGamma),\\
\gamma_n^\pm\circ \mathscr S_\varGamma:H^{-1/2}(\varGamma)&\longrightarrow H^{-1/2}(\varGamma)&\qquad\text{and}\qquad&
\gamma_n^\pm\circ \mathscr D_\varGamma:H^{1/2}(\varGamma)\longrightarrow H^{-1/2}(\varGamma).
\end{alignat*}
Moreover $\gamma_d^+\circ \mathscr S_\varGamma=\gamma_d^-\circ \mathscr S_\varGamma$ (for the single layer 
potential, one-sided Dirichlet traces on 
$\varGamma$ coincide) and
$\gamma_n^+\circ \mathscr D_\varGamma=-\gamma_n^-\circ \mathscr D_\varGamma$ (for 
the double 
layer potential, one-sided Neumann traces on $\varGamma$ have opposite signs). To simplify the notation, we shall drop the superscripts $+$ and $-$ when 
the Dirichlet traces coincide or when the Neumann traces have opposite signs. Thus, we denote $\mathsf S_\varGamma=\gamma_d\circ\mathscr S_\varGamma$ and 
$\mathsf D_\varGamma=\gamma_n\circ\mathscr D_\varGamma$. The operator $\mathsf S_\varGamma:H^{-1/2}(\varGamma)\longrightarrow 
H^{1/2}(\varGamma)$ is an isomorphism 
(recall that the logarithmic capacity of $\varGamma$ 
is assumed to be lower than 1, see \cite[Theorem 8.6]{McLean:2000aa} about this question).
The operator $\mathsf D_\varGamma:H^{1/2}(\varGamma)\longrightarrow 
H^{-1/2}(\varGamma)$ is Fredholm of index 0 with a one dimensional kernel spanned by the function $\mathbf 1_\varGamma$ 
(the constant function equal to one on $\varGamma$)
and with range $\widetilde H^{-1/2}(\varGamma)=\big\{q\in H^{-1/2}(\varGamma)\,:\,\langle q,\mathbf 1_\varGamma\rangle_{-\frac12,\frac12}=0\big\}$. 
Introducing $\widetilde H^{1/2}(\varGamma)=\big\{p\in H^{1/2}(\varGamma)\,:\,(p,\mathbf 1_\varGamma)_{L^2(\varGamma)}=0\big\}$, 
we deduce that $\mathsf D_\varGamma:\widetilde H^{1/2}(\varGamma)\longrightarrow 
\widetilde H^{-1/2}(\varGamma)$ is an isomorphism.
The following identities are usually referred to as the ``jump relations'' on $\varGamma$:
$$\gamma_n^+\circ \mathscr S_\varGamma+\gamma_n^-\circ \mathscr S_\varGamma={\rm Id}\qquad\text{and}\qquad 
\gamma_d^+\circ \mathscr D_\varGamma-\gamma_d^-\circ \mathscr D_\varGamma={\rm Id}.$$
The space $\mathscr A$ of the affine functions in $\mathbb R^2$ 
plays a particular role in the asymptotic behavior of the single layer potential. Indeed, for $|x|$ large, 
the single layer potential admits the following asymptotic expansion:
\begin{subequations}
\label{expand_asymp}
\begin{equation}
\label{asymp_sl}
\mathscr S_\varGamma q(x)=-\frac{1}{2\pi}\langle q,\mathbf 1_\varGamma\rangle_{-\frac12,\frac12} \ln|x|+\frac{1}{2\pi}\frac{x_1}{|x|^2}\langle q,y_1\rangle_{-\frac12,\frac12} +\frac{1}{2\pi}\frac{x_2}{|x|^2}\langle q,y_2\rangle_{-\frac12,\frac12} +
\mathscr O(1/|x|^2).
\end{equation}
The three first terms in the right hand side are not in $L^2(\mathbb R^2)$ while the remainder is.
Let $\mathscr A^{\frac12}_S$ be the three dimensional subspace of $H^{1/2}(\varGamma)$ spanned by the traces of the affine functions. 
Let $\mathscr A^{-\frac12}_S=\mathsf S_\varGamma^{-1}\mathscr A^{\frac12}_S$ (a three dimensional subspace in $H^{-1/2}(\varGamma)$) 
and define $\{\mathsf q_1,\mathsf q_2,\mathsf q_3\}$ a basis of this space normalized in such a way that 
$\langle \mathsf q_j,\mathsf S_\varGamma \mathsf q_k\rangle_{-\frac12,\frac12}=\delta_{j,k}$ (the Kronecker symbol) for every indices 
$j,k\in\{1,2,3\}$. Notice that $\mathscr S_\varGamma\mathsf q_j$ is not an affine function in $\mathbb R^2$ but there exist 
 affine functions $P_j$  such that $\mathscr S_\varGamma\mathsf q_j|_{\varOmega^-}=P_j|_{\varOmega^-}$ ($j=1,2,3$).
\par
Considering now the double layer potential, it can be expanded for $|x|$ large as:
\begin{equation}
\label{asymp_dl}
\mathscr D_\varGamma p(x)=-\frac{1}{2\pi} \frac{x_1}{|x|^2}\langle n_1,p\rangle_{-\frac12,\frac12}
-\frac{1}{2\pi}\frac{x_2}{|x|^2}\langle n_2,p\rangle_{-\frac12,\frac12}+
\mathscr O(1/|x|^2),
\end{equation}
where we recall that $n=(n_1,n_2)$ is the unit normal vector field on $\varGamma$ directed toward the exterior of $\varOmega^-$. 
Let  $\mathscr A^{-\frac12}_D$ be the two dimensional subspace of $\widetilde H^{-1/2}(\varGamma)$ spanned by $n_1$ and $n_2$. Its 
preimage by $\mathsf D_\varGamma$ is a two dimensional subspace of $\widetilde H^{1/2}(\varGamma)$ denoted by $\mathscr A^{\frac12}_D$. 
Let $\{\mathsf p_1,\mathsf p_2\}$ be a basis of this space normalized 
in such a way that $\langle  \mathsf D_\varGamma \mathsf p_j,\mathsf p_k\rangle_{-\frac12,\frac12}=\delta_{j,k}$ ($j,k=1,2$). As for the single 
layer potential, the double layer potential 
$\mathscr D_\varGamma \mathsf p_j$ is not an affine function in $\mathbb R^2$ but there exists 
an affine function $Q_j$ such that $\mathscr D_\varGamma\mathsf p_j|_{\varOmega^-}=Q_j|_{\varOmega^-}$  (for $j=1,2$).
\par
To be complete on the questions of asymptotic behavior of harmonic functions, let us mention a last result borrowed from \cite[Chap. 10, Ex.~1]{Axler:2001aa}. Any   function $v$ harmonic outside a compact set can be expanded in this region as:
\begin{equation}
\label{exp_dfolp}
v(x)=\sum_{j=0}^{+\infty}\mathfrak p_m(x)+\mathfrak q_0\ln|x|+\sum_{j=0}^{+\infty}\frac{\mathfrak q_m(x)}{|x|^{2m}},
\end{equation}
\end{subequations}
where $\mathfrak q_0\in\mathbb R$ and, for every integer $m$, $\mathfrak p_m, \mathfrak q_m$ are harmonic polynomials on $\mathbb R^2$ of degree $m$.
\par
We will mainly rely on the following characterization of the surface potentials in the sequel:
\begin{prop}
\label{prop:def_layer}
Assume that $\varGamma$ is Lipschitz continuous. 
The single layer potential of density $q\in H^{-1/2}(\varGamma)$  is the unique distribution $u\in \mathscr D'(\mathbb R^2)$ satisfying:
\begin{subequations}
\label{def_single}
\begin{alignat}{3}
\label{harm:u}
\langle u,-\Delta\theta\rangle_{\mathscr D'(\mathbb R^2),\mathscr D(\mathbb R^2)}&=\langle q,\gamma_d\theta\rangle_{-\frac12,\frac12}&\quad&\text{for all }\theta\in\mathscr D(\mathbb R^2);\\
\label{asymp:u}
u(x) &= \langle q,\mathbf 1_\varGamma\rangle_{-\frac12,\frac12} G(x)+o(1)&&\text{as}\quad |x|\longrightarrow+\infty.
\end{alignat}
\end{subequations}
The double layer potential of density $p\in H^{1/2}(\varGamma)$ is the unique distribution $v\in \mathscr D'(\mathbb R^2)$ satisfying:
\begin{subequations}
\label{def_double}
\begin{alignat}{3}
\label{harm:v}
\langle v,-\Delta\theta\rangle_{\mathscr D'(\mathbb R^2),\mathscr D(\mathbb R^2)}&=\langle \gamma_n\theta,p\rangle_{-\frac12,\frac12}&\quad&\text{for all }\theta\in\mathscr D(\mathbb R^2);\\
\label{asymp:v}
v(x) &= o(1)&&\text{as}\quad |x|\longrightarrow+\infty.
\end{alignat}
\end{subequations}
\end{prop}
Notice that any distribution $u$ satisfying \eqref{harm:u}  and any distribution $v$ satisfying \eqref{harm:v} is 
harmonic in $\varOmega^+$ so that, according to the generalization to distributions of Weyl's lemma,
they are $\mathcal C^\infty$ in $\varOmega^+$ and the asymptotic conditions \eqref{asymp:u} and \eqref{asymp:v}  make 
sens.
\begin{proof}Let $q$ be in $H^{-1/2}(\varGamma)$. Applying the second Green's identity \eqref{green}, we easily verify that the single layer potential $\mathscr S_\varGamma q$ satisfies both conditions \eqref{def_single}. On the other hand, if $u_1$ and $u_2$ are two distributions 
satisfying these conditions, then $u=u_1-u_2$ is a distribution harmonic in the whole plane. According to Weyl's lemma, it 
is $\mathcal C^\infty$ in $\mathbb R^2$ and since it tends to 0 at infinity, we conclude with Liouville's theorem that $u=0$. The same arguments apply for 
the double layer potential.
\end{proof}
Since in Proposition~\ref{prop:def_layer}, $u=\mathscr S_\varGamma q$ and $v=\mathscr D_\varGamma p$, it turns out that the distribution 
$u$ is actually  in 
$H^1_{\ell oc}(\mathbb R^2)$ while the distribution $v$ is such that $v|_{\varOmega^-}\in H^1(\varOmega^-)$ and 
$v|_{\varOmega^+}\in H^1_{\ell oc}(\overline{\varOmega^+})$.
Our purpose is now to weaken the regularity of $q$ and $p$ and to generalize the definition of the single 
and double layer potentials in order to represent every function in $L^2_{\ell oc}(\mathbb R^2)$, harmonic in $\mathbb R^2\setminus\varGamma$ with  asymptotic behaviors as in \eqref{asymp:u} or \eqref{asymp:v}.
\section{Main function spaces}
\label{main_spaces}
Following an idea of \cite[\S 7]{Amrouche:1994aa}, we introduce the weight fonctions $\rho$ and ${\rm lg}$:
\begin{equation}
\label{weighted}
\rho(x)=\sqrt{1+|x|^2}\qquad\text{and}\qquad {\rm lg}(x)=\ln(2+|x|^2)\forallt x\in\mathbb R^2,
\end{equation}
which enter the definition of the weighted Sobolev space:
$$W^2(\mathbb R^2)=\Big\{u\in \mathscr D'(\mathbb R^2)\,:\, \frac{u}{\rho^2\,{\rm lg}}\in L^2(\mathbb R^2),\, \frac{1}{\rho\,{\rm lg}}\frac{\partial u}{\partial x_j}\in L^2(\mathbb R^2)~\text{ and }~\frac{\partial^2 u}{\partial x_j\partial x_k}\in L^2(\mathbb R^2),\,\forall\,j,k=1,2\Big\}.$$
%
%
\begin{prop}
\label{mqwxf}
The space $W^2(\mathbb R^2)$, provided with its natural norm, enjoys the following properties (borrowed from 
\cite[Theorem 7.2]{Amrouche:1994aa}  for the first and second points and from
\cite[Theorem 9.6]{Amrouche:1994aa} for the third one):
\begin{enumerate}
\item 
The space $\mathscr D(\mathbb R^2)$ is dense in $W^2(\mathbb R^2)$;
\item   There exists a sequence of truncation functions $(\phi_k)_{k\geqslant 1}$ in $\mathscr D(\mathbb R^2)$ 
such that, for every $u\in W^2(\mathbb R^2)$, $\phi_k u\longrightarrow u$ in $W^2(\mathbb R^2)$;
\item The Laplace operator $\Delta:W^2(\mathbb R^2)/\mathscr A\longrightarrow L^2(\mathbb R^2)$
is an isomorphism (we recall that $\mathscr A$ is the space of the affine functions in $\mathbb R^2$).
\end{enumerate}
\end{prop} 
%
For $p\in L^2(\varGamma)$, we denote by $\mu(p)$ the mean value of $p$ on $\varGamma$, i.e. 
 $\mu(p)=|\varGamma|^{-1}(\mathbf 1_\varGamma, p)_{L^2(\varGamma)}$. 
  The original idea at this point is to endow the space 
 $W^2(\mathbb R^2)$ with the following inner products  (for $u,v\in W^2(\mathbb R^2)$):
\begin{subequations}
\label{def:norms}
\begin{align}
\label{scal:S}
(u,v)_S&=(\Delta u,\Delta v)_{L^2(\mathbb R^2)}+
\sum_{j=1}^3\langle  \mathsf q_j,\gamma_d u\rangle_{-\frac12,\frac12}\langle \mathsf q_j,\gamma_d v\rangle_{-\frac12,\frac12}\\
\label{scal:D}
(u,v)_D&=(\Delta u,\Delta v)_{L^2(\mathbb R^2)}+\sum_{j=1}^2( \mathsf p_j,\gamma_n u)_{L^2(\varGamma)}( \mathsf p_j,\gamma_n v)_{L^2(\varGamma)}
+\mu(\gamma_d u)\mu(\gamma_d v),
\end{align}
\end{subequations}
where the subscripts $S$ and $D$ refer to ``single'' (layer) and ``double'' (layer), as it will become clear in the sequel.
The corresponding norms, denoted by $\|\cdot\|_S$ and $\|\cdot\|_D$ are both equivalent to the natural norm of $W^2(\mathbb R^2)$, the proof being a 
straightforward consequence of 
\cite[Corollary 8.4]{Amrouche:1994aa}. It is already worth noting that:
\begin{alignat*}{3}
\mathscr A^\perp&=\big\{u\in W^2(\mathbb R^2)\,:\, (\gamma_du,\gamma_d\theta)_{\frac12}=0\quad\forall\,\theta\in\mathscr A\big\}
&&\text{in }(W^2(\mathbb R^2);\|\cdot\|_S),\\
\mathscr A^\perp&=\big\{u\in W^2(\mathbb R^2)\,:\, (\gamma_nu,\gamma_n\theta)_{L^2(\varGamma)}+
\mu(\gamma_du)\mu(\gamma_d\theta)=0\quad\forall\,\theta\in\mathscr A\big\}
&\quad&\text{in }(W^2(\mathbb R^2);\|\cdot\|_D).
\end{alignat*}
Next, we introduce the boundary spaces:
\begin{equation}
\label{bbghyuj}
{\mathcal H}^{3/2}(\varGamma)=\gamma_d W^2(\mathbb R^2)\qquad\text{and}\qquad {\mathcal H}^{1/2}(\varGamma)=\gamma_n W^2(\mathbb R^2).
\end{equation}
Since the weight functions \eqref{weighted} do not modify the local properties of 
the space, we could as well replace the space $W^2(\mathbb R^2)$ by the space $H^2_{\ell oc}(\mathbb R^2)$ in these definitions. 
We emphasize that the superscripts $3/2$ and $1/2$ in \eqref{bbghyuj} have no other meaning than to recall that 
${\mathcal H}^{3/2}(\varGamma)=
H^{3/2}(\varGamma)$ and ${\mathcal H}^{1/2}(\varGamma)=
H^{1/2}(\varGamma)$ when $\varGamma$ is smooth.
We introduce as well the closed subspaces of 
$W^2(\mathbb R^2)$:
$${W_d^2}(\mathbb R^2)=\{u\in W^2(\mathbb R^2)\,:\, \gamma_du=0\}\qquad\text{and}\qquad 
{W_n^2}(\mathbb R^2)=\{u\in W^2(\mathbb R^2)\,:\, \gamma_nu=0\}.$$
The
 images of ${W_d^2}(\mathbb R^2)$ and ${W_n^2}(\mathbb R^2)$ by $\gamma_n$ and $\gamma_d$  respectively are subspaces of 
 ${\mathcal H}^{3/2}(\varGamma)$ and ${\mathcal H}^{1/2}(\varGamma)$. We denote them by:
\begin{equation}
\label{def:Hronde}
{\mathcal H}^{3/2}_n(\varGamma)=\gamma_dW_n^2(\mathbb R^2)\qquad\text{and}\qquad {\mathcal H}^{1/2}_d(\varGamma)=\gamma_nW^2_d(\mathbb R^2).
\end{equation}
It is well known that when $\varGamma$ is of class $\mathcal C^{1,1}$, the spaces ${\mathcal H}^{3/2}_n(\varGamma)$ and ${\mathcal H}^{3/2}(\varGamma)$ coincide, both being equal to $H^{3/2}(\varGamma)$. In the same way, in the smooth case, ${\mathcal H}^{1/2}_d(\varGamma)={\mathcal H}^{1/2}(\varGamma)=H^{1/2}(\varGamma)$. 
This is no longer true however when $\varGamma$ is a $\mathcal C^{1,1}$ (curvilinear) polygon (and a fortiori when $\varGamma$ 
is only Lipschitz continuous) as explained in \cite{Geymonat:2007wx} where a counterexample 
is provided. Indeed in this case, a pair of functions $(f,g)\in H^{1}(\varGamma)\times L^2(\varGamma)$ is equal to the Dirichlet 
and Neumann traces  of 
a function in $H^2_{\ell oc}(\mathbb R^2)$ if and only if  the vector field $({\partial f}/{\partial\tau})\tau+gn$ is in $H^{1/2}(\varGamma;\mathbb R^2)$.
This condition implies in particular that the functions $f$ and $g$ have to 
satisfy some compatibility conditions 
at the vortices of the domain  (as indicated in \cite[Theorem 1.5.2.4]{Grisvard:1985aa}).
\par
For every $p\in{\mathcal H}^{3/2}(\varGamma)$, we define ${\mathsf L}_d^Sp$ as the unique fonction in $W^2(\mathbb R^2)$ achieving:
\begin{equation}
\label{eq:trala}
\inf\big\{\|u\|_S\,:\,u\in W^2(\mathbb R^2),\,\gamma_du=p\big\}.
\end{equation}
Thus ${\mathsf L}_d^Sp$ is the orthogonal projection of any preimage  of $p$ by $\gamma_d$  on the closed subspace $W^2_d(\mathbb R^2)^\perp$ of 
$(W^2(\mathbb R^2),\|\cdot\|_S)$. It is not difficult to verify that for every $p\in{\mathcal H}^{3/2}(\varGamma)$:
\begin{equation}
\label{ex:LO}
\Delta^2({\mathsf L}_d^Sp)=0\quad\text{ in }\mathscr D'(\mathbb R^2\setminus \varGamma)\quad\text{ and }\quad\gamma_d({\mathsf L}_d^Sp)=p.
\end{equation}
In the same fashion, we define ${\mathsf L}_d^Dp$ by replacing the norm $\|\cdot\|_S$ with the norm $\|\cdot\|_D$ in \eqref{eq:trala}. The function 
 ${\mathsf L}_d^Dp$ verifies both identities \eqref{ex:LO} as well.
This allows us to define two scalar products in ${\mathcal H}^{3/2}(\varGamma)$:
$$( p_1,p_2)_{\frac32}^{A}=\big( \mathsf L_d^Ap_1,\mathsf L_d^A p_2\big)_{\!A}\qquad A\in\{S,D\},$$
whose associated norms, denoted by $\|\cdot\|_{\frac32}^{A}$ are equivalent. The space ${\mathcal H}^{3/2}(\varGamma)$ provided with any of these norms is a Hilbert space.
We denote by $\Pi_d^A$ 
the orthogonal projection onto ${W^2_d}(\mathbb R^2)^\perp$ in $(W^2(\mathbb R^2),\|\cdot\|_A)$. The following identities are obvious:%
\begin{equation}
\label{nice:rel}
\gamma_d\circ \mathsf L_d^A={\rm Id}\qquad\text{and}\qquad \mathsf L_d^A\circ\gamma_d=\Pi_d^A.
\end{equation}
The very same procedure can be carried out by replacing the Dirichlet trace operator $\gamma_d$ with the Neumann 
trace operator $\gamma_n$. This leads us to define 
for $A\in\{S,D\}$
the operators $\mathsf L_n^A$, the projectors $\Pi_n^A$,  the scalar products $(\cdot,\cdot)^{A}_{\frac12}$ and the norms $\|\cdot\|_{\frac12}^{A}$ in the space ${\mathcal H}^{1/2}(\varGamma)$. 
As in \eqref{ex:LO}, the functions $\mathsf L_n^Aq$ verify:
\begin{equation}
\label{ex:L1}
\Delta^2(\mathsf L_n^Aq)=0\quad\text{ in }\mathscr D'(\mathbb R^2\setminus \varGamma)\quad\text{ and }\quad\gamma_n(\mathsf L_n^Aq)=q
\forallt q\in\mathcal H^{\frac12}(\varGamma).
\end{equation}
By construction, the following 
operators are isometric for any $A\in\{S,D\}$:
\begin{subequations}
\label{op:L}
\begin{align}
\label{isom:FRT}
\mathsf L_d^A:\big({\mathcal H}^{3/2}(\varGamma),\|\cdot\|_{\frac32}^{A}\big)\longrightarrow \big({W^2_d}(\mathbb R^2)^\perp,\|\cdot\|_A\big),\\
\label{isom:D}
\mathsf L_n^A:\big({\mathcal H}^{1/2}(\varGamma),\|\cdot\|_{\frac12}^{A}\big)\longrightarrow \big({W^2_n}(\mathbb R^2)^\perp,\|\cdot\|_A\big).
\end{align}
\end{subequations}
The space ${\mathcal H}^{3/2}(\varGamma)$ is continuously embedded in $L^2(\varGamma)$ since there 
exists a constant $C_\varGamma>0$ such that:
$$\|p\|_{L^2(\varGamma)}=\|\gamma_d\circ\mathsf L_d^Sp\|_{L^2(\varGamma)}\leqslant C_\varGamma \|\mathsf L_d^Sp
\|_{W^2(\mathbb R^2)}=\|p\|_{\frac32}^S\forallt p\in {\mathcal H}^{3/2}(\varGamma).$$
The embedding is also dense (because the space 
$\gamma_d\mathscr D(\mathbb R^2)$ is densely embedded in $L^2(\varGamma)$ as claimed in 
Proposition~\ref{prop:den}). Identifying $L^2(\varGamma)$ 
with its dual space by means of Riesz representation theorem, we obtain a so-called Gelfand triple of Hilbert 
spaces (see \cite[Appendix A]{Lequeurre:2020aa}):
\begin{subequations}
\label{gelf}
\begin{equation}
\label{gelfa}
{\mathcal H}^{3/2}(\varGamma)\subset L^2(\varGamma)\subset {\mathcal H}^{-3/2}(\varGamma),
\end{equation}
in which ${\mathcal H}^{-3/2}(\varGamma)$ is the dual space of ${\mathcal H}^{3/2}(\varGamma)$ and $L^2(\varGamma)$ is the pivot space. 
Similarly, we define $\mathcal H^{-1/2}(\varGamma)$ the dual space of $\mathcal H^{1/2}(\varGamma)$ and the Gelfand triple:
\begin{equation}
\label{gelfb}
{\mathcal H}^{1/2}(\varGamma)\subset L^2(\varGamma)\subset {\mathcal H}^{-1/2}(\varGamma).
\end{equation}
\end{subequations}
The Gelfand triples \eqref{gelf} justify that the duality brackets $\llangle\cdot,\cdot\rrangle_{-\frac32,\frac32}$ 
(between  the spaces ${\mathcal H}^{-3/2}(\varGamma)$   and ${\mathcal H}^{3/2}(\varGamma)$) and 
$\llangle\cdot,\cdot\rrangle_{-\frac12,\frac12}$ (between  the spaces ${\mathcal H}^{-1/2}(\varGamma)$   and ${\mathcal H}^{1/2}(\varGamma)$) 
``extend'' the $L^2(\varGamma)$  inner product. Concerning embedding results, we can also state:
\begin{prop}
\label{prop:dens2}
The inclusions $\mathcal H^{3/2}(\varGamma)\subset H^{1/2}(\varGamma)$ and $\mathcal H^{1/2}(\varGamma)
\subset H^{-1/2}(\varGamma)$ are continuous and dense. 
\end{prop}
\begin{proof}
The first inclusion is proved the same way as the inclusion ${\mathcal H}^{3/2}(\varGamma)\subset L^2(\varGamma)$. The second 
inclusion results from the continuity and the density of the inclusion $L^2(\varGamma)\subset H^{-1/2}(\varGamma)$.
\end{proof}
It remains to make precise the topologies of the spaces ${\mathcal H}^{3/2}_n(\varGamma)$ 
and ${\mathcal H}^{1/2}_d(\varGamma)$ introduced in \eqref{def:Hronde}. For every $p\in {\mathcal H}^{3/2}_n(\varGamma)$, we define ${\mathcal L}_np$ as the unique fonction in ${W^2_n}(\mathbb R^2)$ achieving:
\begin{equation}
\label{eq:tralala0}
\inf\big\{\|u\|_D\,:\,u\in {W^2_n}(\mathbb R^2),\,\gamma_du=p\big\}.
\end{equation}
Thus ${\mathcal L}_np$ is the orthogonal projection of any preimage  in   ${W^2_n}(\mathbb R^2)$ of $p$   by $\gamma_d$ 
on the closed subspace $\big({W^2_d}(\mathbb R^2)\cap  {W^2_n}(\mathbb R^2)\big)^\perp$ in the space
$\big({W^2_n}(\mathbb R^2),\|\cdot\|_D\big)$. The function ${\mathcal L}_np$ is biharmonic in $\mathbb R^2\setminus\varGamma$ 
and satisfies $\gamma_d({\mathcal L}_n)p=p$ and  $\gamma_n({\mathcal L}_np)=0$. The space ${\mathcal H}^{3/2}_n(\varGamma)$ is endowed with the inner product:
\begin{equation}
\label{def_norm_10_bis}
(p_1,p_2)_{\frac32,n}=\big({\mathcal L}_np_1,{\mathcal L}_np_2\big)_D=(\Delta{\mathcal L}_np_1,\Delta{\mathcal L}_np_2)_{L^2(\mathbb R^2)}+\mu(p_1)\mu(p_2),\forallt p_1,p_2\in {\mathcal H}^{3/2}_n(\varGamma).
\end{equation}
We denote by $\|\cdot\|_{\frac32,n}$ the corresponding norm. Similarly, for any $q\in {\mathcal H}^{1/2}_d(\varGamma)$, $\mathcal L_dq$ stands for 
the unique function in ${W^2_d}(\mathbb R^2)$ 
achieving:
\begin{equation}
\label{eq:tralala1}
\inf\big\{\|u\|_S\,:\,u\in {W^2_d}(\mathbb R^2),\,\gamma_nu=q\big\}.
\end{equation}
Thus $\mathcal L_dq$ is a function biharmonic in $\mathbb R^2\setminus\varGamma$ 
that satisfies $\gamma_d(\mathcal L_dq)=0$ and  $\gamma_n(\mathcal L_dq)=q$.
The space ${\mathcal H}^{1/2}_d(\varGamma)$ is provided with the scalar product:
\begin{equation}
\label{def_norm_10}
(q_1,q_2)_{\frac12,d}=\big(\mathcal L_dq_1,\mathcal L_dq_2\big)_S=(\Delta\mathcal L_dq_1,\Delta\mathcal L_dq_2)_{L^2(\mathbb R^2)},
\forallt q_1,q_2\in {\mathcal H}^{1/2}_d(\varGamma),
\end{equation}
and the corresponding norm is denoted by $\|\cdot\|_{\frac12,d}$. The spaces $\big({\mathcal H}^{3/2}_n(\varGamma),\|\cdot\|_{\frac32,n}\big)$ 
and $\big({\mathcal H}^{1/2}_d(\varGamma),\|\cdot\|_{\frac12,d}\big)$ are Hilbert spaces and by construction, the following operators are isometric:
\begin{subequations}
\label{op:calL}
\begin{align}
\label{isom:FRT2}
{\mathcal L}_n:\big({\mathcal H}^{3/2}_n(\varGamma),\|\cdot\|_{\frac32,n}\big)&\longrightarrow 
\big(\mathscr B_n(\mathbb R^2),\|\cdot\|_D\big),\\
\label{isom:D2}
\mathcal L_d:\big({\mathcal H}^{1/2}_d(\varGamma),\|\cdot\|_{\frac12,d}\big)&\longrightarrow 
\big(\mathscr B_d(\mathbb R^2),\|\cdot\|_S\big),
\end{align}
\end{subequations}
where $\mathscr B_n(\mathbb R^2)=\big({W^2_d}(\mathbb R^2)\cap {W^2_n}(\mathbb R^2)\big)^\perp\cap {W^2_n}(\mathbb R^2)$
and $\mathscr B_d(\mathbb R^2)=\big({W^2_d}(\mathbb R^2)\cap {W^2_n}(\mathbb R^2)\big)^\perp\cap {W^2_d}(\mathbb R^2)$. The functions 
in $\mathscr B_n(\mathbb R^2)$ are those in $W^2(\mathbb R^2)$ which are biharmonic in $\mathbb R^2\setminus\varGamma$ with homogeneous Neumann 
boundary data and the functions in $\mathscr B_d(\mathbb R^2)$ are biharmonic in $\mathbb R^2\setminus\varGamma$ with homogeneous Dirichlet boundary data.
%
\section{Square integrable surface potentials}
\label{layer_pot}
In this section, we still assume that $\varGamma$ is Lipschitz continuous.
To every $q\in{\mathcal H}^{-3/2}(\varGamma)$ (applying Riesz representation Theorem), we can associate a unique $u_q\in W^2(\mathbb R^2)$ such that:
\begin{subequations}
\begin{equation}
\label{special}
\big(u_q,\theta\big)_S=\llangle q,\gamma_d\theta\rrangle_{-\frac32,\frac32}\forallt \theta\in W^2(\mathbb R^2),
\end{equation}
and we define:
\begin{equation}
\mathscr S_\varGamma^\dagger q=-\Delta u_q+\sum_{j=1}^3
\langle \mathsf q_j,\gamma_d u_q\rangle_{-\frac12,\frac12}\mathscr S_\varGamma \mathsf q_j.
\end{equation}
\end{subequations}
Similarly, to every $p\in{\mathcal H}^{-1/2}(\varGamma)$, we can associate a unique $v_p\in W^2(\mathbb R^2)$ such that:
\begin{subequations}
\begin{equation}
\label{special_1}
\big(v_p,\theta\big)_D=\llangle p,\gamma_n\theta\rrangle_{-\frac12,\frac12}\forallt \theta\in W^2(\mathbb R^2),
\end{equation}
and we define:
\begin{equation}
\mathscr D_\varGamma^\dagger p=-\Delta v_p+\sum_{j=1}^2(\mathsf p_j,\gamma_n v_p)_{L^2(\varGamma)}\mathscr D_\varGamma \mathsf p_j.\end{equation}
\end{subequations}
The expressions of the functions $u_q$ and $v_q$ with respect to $q$ and $p$ can be made precise.
Considering the Gelfand triple \eqref{gelfa} and  \eqref{gelfb}, we can classically  (see \cite[Appendix A]{Lequeurre:2020aa}) 
define the isometric operators 
\begin{equation}
\label{def_gelf_op}
\fct{\mathsf T_{\! d}:{\mathcal H}^{3/2}(\varGamma)}{\mathcal H^{-3/2}(\varGamma)}
{p}{(p,\cdot)_{\frac32}^S}
\qquad\text{and}\qquad
\fct{\mathsf T_{\! n}:{\mathcal H}^{1/2}(\varGamma)}{\mathcal H^{-1/2}(\varGamma)}
{q}{(q,\cdot)_{\frac12}^D.}
\end{equation}
\begin{lemma}
\label{lem:234}
For every $q\in{\mathcal H}^{-3/2}(\varGamma)$, the function $u_q$ defined by \eqref{special} is 
equal to ${\mathsf L}_d^S\circ \mathsf T_{\! d}^{-1}q$. For every $p\in{\mathcal H}^{-1/2}(\varGamma)$, the function $v_p$ defined by 
\eqref{special_1} is 
equal to ${\mathsf L}_n^D\circ \mathsf T_{\! n}^{-1}p$. It follows that the applications:
$$\fct{{\mathcal H}^{-3/2}(\varGamma)}{\big({W^2_d}(\mathbb R^2)^\perp,\|\cdot\|_S\big)}{q}{u_q}
\qquad\text{and}\qquad
\fct{{\mathcal H}^{-1/2}(\varGamma)}{\big({W^2_n}(\mathbb R^2)^\perp,\|\cdot\|_D\big)}{p}{v_p,}
$$
are isometric and that:
\begin{subequations}
\label{bounded_SD}
\begin{alignat}{3}
\label{bounded_SD1}
\mathscr S_\varGamma^\dagger q&=-\Delta {\mathsf L}_d^S\circ \mathsf T_{\! d}^{-1}q+\sum_{j=1}^3
\langle \mathsf q_j,\mathsf T_{\! d}^{-1}q\rangle_{-\frac12,\frac12}\mathscr S_\varGamma \mathsf q_j&&\forallt q\in \mathcal H^{-3/2}(\varGamma),\\
\label{bounded_SD2}
\mathscr D_\varGamma^\dagger p&=-\Delta {\mathsf L}_n^D\circ \mathsf T_{\! n}^{-1}p+\sum_{j=1}^2
\langle\mathsf T_{\! n}^{-1}p, \mathsf p_j\rangle_{-\frac12,\frac12}\mathscr D_\varGamma \mathsf p_j
&&\forallt p\in \mathcal H^{-1/2}(\varGamma).
\end{alignat}
\end{subequations}
\end{lemma} 
%
\begin{proof}
Let $q\in{\mathcal H}^{-3/2}(\varGamma)$ and $\theta\in W^2(\mathbb R^2)$. By definition of the operator $\mathsf T_{\! d}$:
$$\llangle q,\gamma_d\theta\rrangle_{-\frac32,\frac32}=\big(\mathsf T_{\! d}^{-1}q,\gamma_d\theta\big)_{\frac32}^S=\big({\mathsf L}_d^S\circ 
\mathsf T_{\! d}^{-1}q,{\mathsf L}_d^S\circ\gamma_d\theta\big)_S.$$
According to \eqref{nice:rel}:
$$\big({\mathsf L}_d^S\circ 
\mathsf T_{\! d}^{-1}q,{\mathsf L}_d^S\circ\gamma_d\theta\big)_S=\big({\mathsf L}_d^S\circ 
\mathsf T_{\! d}^{-1}q,\Pi_d^S\theta\big)_S=\big({\mathsf L}_d^S\circ 
\mathsf T_{\! d}^{-1}q,\theta\big)_S,$$
which means that $u_q={\mathsf L}_d^S\circ \mathsf T_{\! d}^{-1}q$ considering \eqref{special}. 
The result concerning $v_p$ is proved in the same way.
\end{proof}
\begin{theorem}
\label{exten_layer}
The linear operators $\mathscr S_\varGamma^\dagger:{\mathcal H}^{-3/2}(\varGamma)\longrightarrow L^2_{\ell oc}(\mathbb R^2)$ and
$\mathscr D_\varGamma^\dagger:{\mathcal H}^{-1/2}(\varGamma)\longrightarrow L^2_{\ell oc}(\mathbb R^2)$ are bounded and they
satisfy:
\begin{itemize}
\item[--] For every $q\in{\mathcal H}^{-3/2}(\varGamma)$:
\begin{subequations}
\label{def_single_ext}
\begin{alignat}{3}
\label{harm:u_ext}
\big\langle \mathscr S_\varGamma^\dagger q,-\Delta\theta\big\rangle_{\mathscr D'(\mathbb R^2),\mathscr D(\mathbb R^2)}&=\llangle q,\gamma_d\theta\rrangle_{-\frac32,\frac32}&\quad&\text{for all }\theta\in\mathscr D(\mathbb R^2);\\
\label{asymp:u_ext}
\mathscr S_\varGamma^\dagger q(x) &= \llangle q,\mathbf 1_\varGamma\rrangle_{-\frac32,\frac32} G(x)+o(1)&&\text{as}\quad |x|\longrightarrow+\infty;
\end{alignat}
\end{subequations}
\item[--] For every $p\in {\mathcal H}^{-1/2}(\varGamma)$:
\begin{subequations}
\label{def_double_ext}
\begin{alignat}{3}
\label{harm:v_ext}
\langle \mathscr D_\varGamma^\dagger p,-\Delta\theta\rangle_{\mathscr D'(\mathbb R^2),\mathscr D(\mathbb R^2)}&=\llangle p,\gamma_n\theta\rrangle_{-\frac12,\frac12}&\quad&\text{for all }\theta\in\mathscr D(\mathbb R^2);\\
\label{asymp:v_ext}
\mathscr D_\varGamma^\dagger p(x) &= o(1)&&\text{as}\quad |x|\longrightarrow+\infty.
\end{alignat}
\end{subequations}
\end{itemize}
The operators $\mathscr S_\varGamma^\dagger$ and $\mathscr D_\varGamma^\dagger$ are the extensions by density of the classical 
single and double layer potentials to the spaces ${\mathcal H}^{-3/2}(\varGamma)$ and ${\mathcal H}^{-1/2}(\varGamma)$ respectively.
\end{theorem}
\begin{proof}
For every $q\in\mathcal H^{-3/2}(\varGamma)$, we can rewrite \eqref{special}:
$$(\Delta u_q,\Delta \theta)_{L^2(\mathbb R^2)}+
\sum_{j=1}^3\langle \mathsf q_j,\gamma_d u_q\rangle_{-\frac12,\frac12}\langle \mathsf q_j,\gamma_d \theta\rangle_{-\frac12,\frac12}=
\llangle q,\gamma_d\theta\rrangle_{-\frac32,\frac32}\forallt \theta\in\mathscr D(\mathbb R^2).$$
According to \eqref{harm:u}, we can transform the second term in the left and side to obtain:
$$\Big\langle\Delta u_q-\sum_{j=1}^3\langle \mathsf q_j,\gamma_d u_q\rangle_{-\frac12,\frac12}\mathscr S_\varGamma\mathsf  q_j,\Delta \theta\Big\rangle_{\mathscr D'(\mathbb R^2),
\mathscr D(\mathbb R^2)}=\llangle q,\gamma_d\theta\rrangle_{-\frac32,\frac32}\forallt \theta\in\mathscr D(\mathbb R^2),$$
which is \eqref{harm:u_ext}.
Let now $x$ be a point in $\varOmega^+$ and denote by $d(x,\varGamma)$ the distance from $x$ to $\varGamma$.
On the disk $D(x,d(x,\varGamma))$ of center $x$ and radius $d(x,\varGamma)$, the function $\Delta u_q$ is harmonic. It follows that:
$$\Delta u_q(x)=-\mathscr S^\dagger_\varGamma q(x)+\sum_{k=1}^3\langle q_k,\gamma_d u_q\rangle_{-\frac12,\frac12}
\mathscr S_\varGamma \mathsf q_k(x)
=\frac{1}{\pi d(x,\varGamma)^2}\int_{D(x,d(x,\varGamma))}\Delta u_q(y)\dy,$$
from which we deduce that:
\begin{equation}
\label{tghujio}
\left|\mathscr S^\dagger_\varGamma q(x)-\sum_{k=1}^3\langle \mathsf q_k,\gamma_d u_q\rangle_{-\frac12,\frac12}\mathscr S_\varGamma 
\mathsf q_k(x)\right|
\leqslant \frac{1}{\sqrt{\pi} d(x,\varGamma)}\|\Delta u_q\|_{L^2(\mathbb R^2)}.
\end{equation}
Taking into account the asymptotic expansion \eqref{asymp_sl}, we obtain on the one hand:
\begin{subequations}
\label{sub:1}
\begin{equation}
\sum_{k=1}^3\langle \mathsf q_k,\gamma_d u_q\rangle_{-\frac12,\frac12}\mathscr S_\varGamma \mathsf q_k(x)= \left(
\sum_{k=1}^3\langle \mathsf q_k,\gamma_d u_q\rangle_{-\frac12,\frac12}\langle \mathsf q_k,\mathbf 1_\varGamma\rangle_{-\frac12,\frac12}\right)G(x)+\mathscr O(1/|x|).
\end{equation}
On the other hand, equality \eqref{special} with $\theta=\mathbf 1_{\mathbb R^2}$ yields:
\begin{equation}
\llangle q,\mathbf 1_{\varGamma}\rrangle_{-\frac32,\frac32}=\sum_{k=1}^3\langle \mathsf q_k,\gamma_d u_q\rangle_{-\frac12,\frac12}\langle \mathsf q_k,\mathbf 1_\varGamma\rangle_{-\frac12,\frac12}.
\end{equation}
\end{subequations}
Combining both equations \eqref{sub:1} with \eqref{tghujio} and letting $|x|$ go to $+\infty$, we obtain 
 \eqref{asymp:u_ext}. The proof of   equalities \eqref{def_double_ext} is similar. The only difficulty consists in noticing that the function 
 $v_p$ in \eqref{special_1} achieves:
 $$\min_{v\in W^2(\mathbb R^2)}\frac{1}{2}\|v\|_D^2-\llangle p,\gamma_n v\rrangle_{-\frac12,\frac12},$$
 and therefore that $\mu(v_p)=0$.
 \par
According to \eqref{gelfa} 
 and Propositions~\ref{prop:dens2}, all the following 
 inclusions are continuous and dense:
 $$\mathcal H^{3/2}(\varGamma)\subset H^{1/2}(\varGamma)\subset L^2(\varGamma)\subset H^{-1/2}(\varGamma)\subset \mathcal H^{-3/2}(\varGamma).$$
 It entails that for every $q\in H^{-1/2}(\varGamma)$ and $p\in \mathcal H^{3/2}(\varGamma)$, we are allowed to write:
 $$\llangle q,p\rrangle_{-\frac32,\frac32} =\langle q,p\rangle_{-\frac12,\frac12}.$$
Comparing \eqref{def_single} and \eqref{def_single_ext}, we conclude that   $\mathscr S_\varGamma^\dagger q=\mathscr S_\varGamma q$
for every $q\in H^{-1/2}(\varGamma)$. In the same fashion, we can prove that  $\mathscr D_\varGamma^\dagger p=\mathscr D_\varGamma p$
for every $p\in   H^{1/2}(\varGamma)$. 
It remains only to verify that $\mathscr S^\dagger_\varGamma$ and $\mathscr D^\dagger_\varGamma$ are bounded but this is a 
straightforward consequence of the expressions \eqref{bounded_SD}  in Lemma~\ref{lem:234}.
\end{proof}

\section{Further function spaces}
\label{sec:polyn}
In this section, we assume that $\varGamma$ is a $\mathcal C^{1,1}$ curvilinear polygon and we
denote by $\varGamma_j$ its $\mathcal C^{1,1}$ edges and by $c_j$ its vertices ($j=1,\ldots,N$).
In the sequel we will need some particular test functions in $W^2(\mathbb R^2)$. Their existence is asserted in the Lemma below:
\begin{lemma}
\label{lem:1}
\begin{enumerate}
\item Any function $\theta$ in $W^2(\mathbb R^2)$ supported in $\mathbb R^2\setminus\{c_1,\ldots,c_N\}$ can be decomposed 
into a sum of two functions $\theta_d+\theta_n$ with $\theta_d\in {W^2_d}(\mathbb R^2)$ and $\theta_n\in {W^2_n}(\mathbb R^2)$.
\item For every index $k\in\{1,\ldots,N\}$, there exists a function $\theta$ in ${W^2_n}(\mathbb R^2)$ such that $\theta(c_k)=1$ and $\theta(c_j)=0$ when 
$j\neq k$, $j\in\{1,\ldots,N\}$.
\end{enumerate}
\end{lemma}
\begin{proof}
Addressing the first point of the lemma, we  apply \cite[Theorem 1.5.2.4]{Grisvard:1985aa} which makes precise the range 
of the operator $(\gamma_d^-,\gamma_n^-)$ defined on the space $H^2(\varOmega^-)$. Since, in 
\cite[Theorem 10.4.1]{Agranovich:2015us}, the author 
proves the existence of a universal extension operator from $H^2(\varOmega^-)$ to $H^2(\mathbb R^2)$, 
the range of $(\gamma_d^-,\gamma_n^-)$  is the same when we consider 
this operator as defined on $H^2_{\ell oc}(\mathbb R^2)$ or on $W^2(\mathbb R^2)$. So let $\theta$ be given in $W^2(\mathbb R^2)$ and 
denote respectively by $f_j$ and $g_j$ the restrictions of $\gamma_d^-\theta$ and $\gamma_n^-\theta$ 
to the edge $\varGamma_j$ (for $j$ ranging from 1 to $N$). According to \cite[Theorem 1.5.2.8]{Grisvard:1985aa},  the pair $(f_j,g_j)$ 
belongs to the space $H^{3/2}(\varGamma_j)\times H^{1/2}(\varGamma_j)$ for every index $j=1,\ldots,N$. Considering now 
the pairs $(f_j,0)$ in the same space $H^{3/2}(\varGamma_j)\times H^{1/2}(\varGamma_j)$, they trivially satisfy the compatibility conditions at 
the vertices $c_k$ described in \cite[Theorem 1.5.2.8]{Grisvard:1985aa} since every fonction $f_j$ is compactly supported on $\varGamma_j$.
Therefore they belong to the range of $(\gamma^-_d,\gamma^-_n)$ and there exists a function $\theta_n$ in $W^2(\mathbb R^2)$ such 
that $\gamma_d^-\theta_n|_{\varGamma_j}=\gamma_d\theta_n|_{\varGamma_j}=f_j$ and $\gamma_n^-\theta_n|_{\varGamma_j}=\gamma_n\theta_n|_{\varGamma_j}=0$. We define $\theta_d=\theta-\theta_n$ and the former assertion of the lemma is proved.
\par
The proof of the latter rests roughly on the same arguments. Let $k$ be given in $\{1,\ldots,N\}$ and let $f$ be a smooth function defined on $\varGamma$ that vanishes on a neighborhood 
of every vertex $c_j$ when $j\neq k$ and is constant in a neighborhood of $c_k$. Denote by $f_j$ the restriction of $f$ to $\varGamma_j$ ($j=1,\ldots,N$). 
The pairs $(f_j,0)$ belong to $H^{3/2}(\varGamma_j)\times H^{1/2}(\varGamma_j)$ and they trivially satisfy the compatibility conditions at 
the vertices described in \cite[Theorem 1.5.2.8]{Grisvard:1985aa} (since $\partial f_j/\partial\tau$ vanishes near the vertices),
what ensures the existence in $W^2(\mathbb R^2)$ of a preimage $\theta$ by the operator $(\gamma_d,\gamma_n)$. 
\end{proof}
Recall that the spaces ${\mathcal H}^{3/2}_n(\varGamma)$ and ${\mathcal H}^{1/2}_d(\varGamma)$ are defined in 
\eqref{def:Hronde}. The following result will play an important role in the rest of the paper:
\begin{theorem}
\label{dense:prop}
The space ${\mathcal H}^{3/2}_n(\varGamma)$ is dense 
in ${\mathcal H}^{3/2}(\varGamma)$
and the space
${\mathcal H}^{1/2}_d(\varGamma)$ is dense in ${\mathcal H}^{1/2}(\varGamma)$.
\end{theorem}
\begin{proof}The proofs of both assertions are similar so let us focus on the latter.
Using the isometric operator \eqref{isom:D} and since $\mathsf L_n^D\circ\gamma_n=\Pi_n^D$,
 we are led to prove 
that $\Pi_n^D {W^2_d}(\mathbb R^2)$ is dense in $\Pi_n^D W^2(\mathbb R^2)=W_n^2(\mathbb R^2)^\perp$. This is equivalent 
to showing that $W^2_d(\mathbb R^2)\oplus W_n^2(\mathbb R^2)$ is dense in $W^2(\mathbb R^2)$ or, still equivalently, that 
$W^2_d(\mathbb R^2)^\perp\cap W_n^2(\mathbb R^2)^\perp=\{0\}$ (where both superscripts $\perp$  refer to the same 
scalar product $(\cdot,\cdot)_D$). 
\par
So, let $u$ be in $W^2_d(\mathbb R^2)^\perp$. Then $u=\Pi_d^Du$ and therefore:
\begin{subequations}
\begin{equation}
\label{lala}
(u,\theta)_D=\big(\Pi_d^D u,\theta\big)_D=\big(\Pi_d^D u,\Pi_d^D\theta\big)_D=\big({\mathsf L}_d^D\circ\gamma_du,{\mathsf L}_d^D\circ\gamma_d\theta\big)_D=
(\gamma_du,\gamma_d\theta)_{\frac32}^{D}\forallt\,\theta\in W^2(\mathbb R^2).
\end{equation}
In the same fashion, assuming that the function $u$ belongs also to $W_n^2(\mathbb R^2)^\perp$ we get:
\begin{equation}
\label{lolo}
(u,\theta)_D=
(\gamma_nu,\gamma_n\theta)_{\frac12}^{D}\forallt\,\theta\in W^2(\mathbb R^2).
\end{equation}
\end{subequations}
In addition, $u$ achieves:
$$\inf\big\{\|v\|_D\,:\,v\in W^2(\mathbb R^2),\,\gamma_nv=\gamma_nu\big\},$$
and therefore $\mu(u)=0$.
Now, recall that $\{c_1,\ldots,c_N\}$ are the vertices of the polygon $\varGamma$. According to Lemma~\ref{lem:1}, every function $\theta$ in $W^2(\mathbb R^2)$ 
compactly supported in $\mathbb R^2\setminus\{c_1,\ldots,c_N\}$ can be decomposed into the sum 
of two functions $\theta_d\in {W^2_d}(\mathbb R^2)$ and $\theta_n\in {W^2_n}(\mathbb R^2)$. It is easy to verify that  both 
functions can be chosen compactly supported in $\mathbb R^2\setminus\{c_1,\ldots,c_N\}$. If follows that for such a function 
$\theta$, we have:
$$(u,\theta)_D=(u,\theta_d)_D+(u,\theta_n)_D=0,$$
where we have used \eqref{lala} for the former term in the right hand side and \eqref{lolo} for the latter. Thus, we have proved in particular that:
$$(\Delta u,\Delta \theta)_{L^2(\mathbb R^2)}+\sum_{j=1}^2(\mathsf p_j,\gamma_n u)_{L^2(\varGamma)}(\mathsf p_j,\gamma_n \theta)_{L^2(\varGamma)}=0
\forallt \theta\in\mathscr D(\mathbb R^2\setminus\{c_1,\ldots,c_N\}),$$
and this can be rewritten, according to \eqref{harm:v} as:
$$\Big\langle-\Delta u+\sum_{j=1}^2(\mathsf p_j,\gamma_n u)_{L^2(\varGamma)}\mathscr D_\varGamma \mathsf p_j,\Delta \theta\Big\rangle_{\mathscr D'(\mathbb R^2),
\mathscr D(\mathbb R^2)}=0
\forallt \theta\in\mathscr D(\mathbb R^2\setminus\{c_1,\ldots,c_N\}).$$
This equality means that the function
\vspace{-1mm}
\begin{equation}
\label{wwqsdrt}
v=-\Delta u+\sum_{j=1}^2(\mathsf p_j,\gamma_n u)_{L^2(\varGamma)}\mathscr D_\varGamma \mathsf p_j,
\end{equation}
is harmonic in $\mathbb R^2\setminus\{c_1,\ldots,c_N\}$ and the distribution $\Delta v$ is  supported 
in the points $c_1,\ldots,c_N$. According to \cite[Theorem 1.5.3]{Hormander:1976uu},  we deduce that this distribution is a finite linear 
combination of Dirac measures and derivatives of Dirac mesures at the points $c_j$ ($j=1,\ldots,N$). Derivatives of Dirac measures 
must be excluded however since $v$ is in 
$L^2_{\ell oc}(\mathbb R^2)$. Finally, $v$ can only take the form:
\begin{equation}
\label{gugu}
v = -\sum_{j=1}^n \frac{\varrho_j}{2\pi} \ln|\cdot-c_j|+h,
\end{equation}
with $\varrho_j\in\mathbb R$ and $h$ harmonic in $\mathbb R^2$. Proceeding as in the proof of Theorem~\ref{exten_layer}, 
we deduce from identity \eqref{wwqsdrt} that $v(x)=o(1)$ as $|x|\longrightarrow+\infty$. It follows that $\sum_{j= 1}^n\varrho_j=0$ 
and $h=0$ with Liouville's theorem. Let 
$k\in\{1,\ldots,N\}$ be given and  let $\theta$ be 
a function in ${W^2_n}(\mathbb R^2)$ compactly supported such that $\theta(c_j)=0$ for $j\neq k$ and $\theta(c_k)=1$. Such a function exists 
according to Lemma~\ref{lem:1} and yields, applying   Green's identity \eqref{green}:
$$(u,\theta)_D=(\Delta u,\Delta \theta)_{L^2(\mathbb R^2)}+\sum_{j=1}^2(\mathsf p_j,\gamma_n u)_{L^2(\varGamma)}(\mathsf p_j,\gamma_n \theta)_{L^2(\varGamma)}=
(v,\Delta\theta)_{L^2(\mathbb R^2)}.$$
Using the expression \eqref{gugu} of the function $v$, we classically obtain that $(u,\theta)_D=\varrho_k$.
On the other hand, identity \eqref{lolo} leads to $(u,\theta)_D=0$, what completes the proof.
\end{proof}
Considering back the Gelfand triples \eqref{gelf}, we are allowed to write when $\varGamma$ is a curvilinear $\mathcal C^{1,1}$ 
polygon:
\begin{subequations}
\label{chain_of_inclusions}
\begin{equation}
{\mathcal H}^{3/2}_n(\varGamma)\subset {\mathcal H}^{3/2}(\varGamma)\subset H^{1/2}(\varGamma)\subset L^2(\varGamma)
\subset H^{-1/2}(\varGamma)\subset {\mathcal H}^{-3/2}(\varGamma)\subset {\mathcal H}^{-3/2}_n(\varGamma),
\end{equation}
where all the inclusions are continuous and dense, $L^2(\varGamma)$ is the pivot space (i.e. the space identified via Riesz 
representation theorem with its dual space) and ${\mathcal H}^{-3/2}_n(\varGamma)$ is the dual space of ${\mathcal H}^{3/2}_n(\varGamma)$. In a similar 
way, we have also:
\begin{equation}
\label{second_chain}
{\mathcal H}^{1/2}_d(\varGamma)\subset {\mathcal H}^{1/2}(\varGamma)\subset L^2(\varGamma)
\subset {\mathcal H}^{-1/2}(\varGamma)\subset {\mathcal H}^{-1/2}_d(\varGamma).
\end{equation}
\end{subequations}
We denote respectively by $\llangle\cdot,\cdot\rrangle_{-\frac32,\frac32,n}$ and $\llangle\cdot,\cdot\rrangle_{-\frac12,\frac12,d}$ the 
duality pairings on $\mathcal H^{-3/2}_n(\varGamma)\times\mathcal H^{3/2}_n(\varGamma)$ and on
$\mathcal H^{-1/2}_d(\varGamma)\times\mathcal H^{1/2}_d(\varGamma)$ and we introduce the isometric operators, based on the 
Gelfand triple structure:
\begin{equation}
\label{tD}
\fct{\mathcal T_{\! d}:{\mathcal H}^{1/2}_d(\varGamma)}{\mathcal H^{-1/2}_d(\varGamma)}
{q}{(q,\cdot)_{\frac12,d},}
\qquad\text{and}\qquad
\fct{\mathcal T_{\! n}:{\mathcal H}^{3/2}_n(\varGamma)}{\mathcal H^{-3/2}_n(\varGamma)}
{p}{(p,\cdot)_{\frac32,n}.}
\end{equation}
\par
We end this section by defining the closed subspaces of $L^2(\mathbb R^2)$ consisting in functions that are harmonic in $\varOmega^+\cup \varOmega^-$:
$$\mathscr H^0(\mathbb R^2\setminus\varGamma)=\big\{u\in L^2(\mathbb R^2)\,:\, (u,\Delta\theta)_{L^2(\mathbb R^2)}=0,
\quad\forall\,\theta\in\mathscr D(\mathbb R^2\setminus\varGamma)\big\}.$$
Combining Proposition~\ref{mqwxf} with Theorem~\ref{kernel}, it follows that: 
\begin{equation}
\label{def_harmL2}
\mathscr H^0(\mathbb R^2\setminus\varGamma)=\big\{\Delta u\,:\,u\in\big({W^2_d}(\mathbb R^2)\cap {W^2_n}(\mathbb R^2)\big)^\perp\big\},
\end{equation}
where the superscript $\perp$ refers to any of the two scalar products \eqref{def:norms} 
defined on $W^2(\mathbb R^2)$ (both leading to the 
same space). We will also consider more regular harmonic functions:
$$\mathscr H^1(\mathbb R^2\setminus\varGamma)=\big\{u\in L^2(\mathbb R^2)\,:\,u|_{\varOmega^+}\in H^1(\varOmega^+)\text{ and }
u|_{\varOmega^-}\in H^1(\varOmega^-)\big\}.$$
\begin{prop}
The space $\mathscr H^1(\mathbb R^2\setminus\varGamma)$ is dense in $\mathscr H^0(\mathbb R^2\setminus\varGamma)$.
\end{prop}
\begin{proof}
We introduce the closed subspace of $H^{-1/2}(\varGamma)\times H^{1/2}(\varGamma)$:
$$E(\varGamma)=\big\{(q,p)\in H^{-1/2}(\varGamma)\times H^{1/2}(\varGamma)\,:\, 
\langle q,\mathbf 1_\varGamma\rangle_{-\frac12,\frac12}=0,\,
\langle q,y_k\rangle_{-\frac12,\frac12}-\langle n_k,p\rangle_{-\frac12,\frac12}=0,\,k=1,2\big\}.$$
We claim that:
\begin{equation}
\label{ghaqpl}
\mathscr H^1(\mathbb R^2\setminus\varGamma)=\big\{\mathscr S_\varGamma q+\mathscr D_\varGamma p\,:\,(q,p)\in E(\varGamma)\big\}.
\end{equation}
For any $(q,p)\in E(\varGamma)$, the function $u=\mathscr S_\varGamma q+\mathscr D_\varGamma p$ is in
 $\mathscr H^1(\mathbb R^2\setminus\varGamma)$ according to the asymptotic expansions \eqref{expand_asymp}.
Reciprocally, let $u$ be in $\mathscr H^1(\mathbb R^2\setminus\varGamma)$ and denote $q=\gamma_n^+u+\gamma_n^-u$ 
and $p=\gamma_d^+u-\gamma_d^-u$. The function $v=\mathscr S_\varGamma q+\mathscr D_\varGamma p-u$ is harmonic on
$\mathbb R^2$. Since, by hypothesis, $u\in L^2(\mathbb R^2)$ we can proceed as in 
the proof of Proposition~\ref{exten_layer} to show that $u(x)=\mathscr O(1/|x|)$ for $|x|$ large. 
Taking into account \eqref{expand_asymp} again, we deduce that:
$$v(x)=-\frac{1}{2\pi}\langle q,\mathbf 1_\varGamma\rangle_{-\frac12,\frac12} \ln|x|+\mathscr O(1/|x|)
\qquad\text{as }|x|\longrightarrow+\infty,$$
and invoking \cite[Theorem 9.10]{Axler:2001aa}, we conclude that the function $v$ vanishes on $\mathbb R^2$. It follows that 
$u$ is equal to $\mathscr S_\varGamma q+\mathscr D_\varGamma p$ and since this function is in $L^2(\mathbb R^2)$, the pair $(q,p)$ 
is in $E(\varGamma)$ and identity \eqref{ghaqpl} is proved. We can now determine the space 
$\mathscr H^1(\mathbb R^2\setminus\varGamma)^\perp$ in 
$\mathscr H^0(\mathbb R^2\setminus\varGamma)$. Let $w$ be in $\mathscr H^0(\mathbb R^2\setminus\varGamma)$ such that:
$$(w,u)_{L^2(\mathbb R^2)}=0\forallt u\in \mathscr H^1(\mathbb R^2\setminus\varGamma),$$
or equivalently:
$$\big(w,\mathscr S_\varGamma q+\mathscr D_\varGamma p\big)_{L^2(\mathbb R^2)}=0\forallt (q,p)\in E(\varGamma).$$
From \eqref{def_harmL2}, we know that there exists a function $v\in \big({W^2_d}(\mathbb R^2)\cap {W^2_n}(\mathbb R^2)\big)^\perp$ 
such that $w=\Delta v$. Following Proposition~\ref{mqwxf}, $\mathscr D(\mathbb R^2)$ is a dense subspace 
in $W^2(\mathbb R^2)$ so let $(v_k)_{k\geqslant 0}$ be a sequence in $\mathscr D(\mathbb R^2)$ that converges to $v$ in $W^2(\mathbb R^2)$. 
For every $k\geqslant 0$ we can apply the second Green's formula \eqref{green} to obtain:
$$\langle \gamma_n v_k,p\rangle_{-\frac12,\frac12}-\langle q,\gamma_d v_k\rangle_{-\frac12,\frac12}=
\big(-\Delta v_k,\mathscr S_\varGamma q+\mathscr D_\varGamma p\big)_{L^2(\mathbb R^2)}.$$
 Letting $k$ 
go to $+\infty$, it comes:
$$\langle \gamma_n v,p\rangle_{-\frac12,\frac12}-\langle q,\gamma_d v\rangle_{-\frac12,\frac12}=
\big(-\Delta v,\mathscr S_\varGamma q+\mathscr D_\varGamma p\big)_{L^2(\mathbb R^2)}=0,$$
 and since the inclusions 
$H^{1/2}(\varGamma)
\subset \mathcal H^{-1/2}(\varGamma)$ and $H^{-1/2}(\varGamma)\subset \mathcal H^{-3/2}(\varGamma)$ are dense
(see \eqref{chain_of_inclusions}), 
we deduce that:
\begin{subequations}
\label{fgtyhuj}
\begin{equation}
\llangle p,\gamma_n v\rrangle_{-\frac12,\frac12}-\llangle q,\gamma_d v\rrangle_{-\frac32,\frac32}=0,
\end{equation}
for every $(q,p)\in\mathcal H^{-3/2}(\varGamma)\times \mathcal H^{-1/2}(\varGamma)$ such that:
\begin{equation}
\llangle q,\mathbf 1_\varGamma\rrangle_{-\frac32,\frac32}=0\quad\text{and}\quad
\llangle q,y_k\rrangle_{-\frac32,\frac32}-\llangle p,n_k\rrangle_{-\frac12,\frac12}=0\quad(k=1,2).
\end{equation}
\end{subequations}
Notice that, for $k=1,2$, $y_k$ and $n_k$ are the Dirichlet and Neumann traces of affine functions and therefore that they are respectively 
in $\mathcal H^{3/2}(\varGamma)$ and $\mathcal H^{1/2}(\varGamma)$. Both equalities \eqref{fgtyhuj} mean that there exist three 
real numbers $\lambda_1,\lambda_2$ and $\lambda_3$ such that:
$$\begin{pmatrix}
\gamma_d v\\
\gamma_n v
\end{pmatrix}
=\lambda_1
\begin{pmatrix}y_1\\n_1\end{pmatrix}
+\lambda_2
\begin{pmatrix}y_2\\n_2\end{pmatrix}
+\lambda_3
\begin{pmatrix}\mathbf 1_\varGamma\\0\end{pmatrix}.$$
We deduce that $v$ 
minus a linear combination of affine functions is in ${W^2_d}(\mathbb R^2)\cap {W^2_n}(\mathbb R^2)$. 
But since $\mathscr A\subset \big({W^2_d}(\mathbb R^2)\cap {W^2_n}(\mathbb R^2)\big)^\perp$ and 
$v\in  \big({W^2_d}(\mathbb R^2)\cap {W^2_n}(\mathbb R^2)\big)^\perp$, this function is also in
$\big({W^2_d}(\mathbb R^2)\cap {W^2_n}(\mathbb R^2)\big)^\perp$. If follows that $v\in\mathscr A$ and $w=\Delta v=0$, which concludes 
the proof.
\end{proof}
\section{Traces and jump relations}
\label{SEC:jump}
 According to Theorem~\ref{dirichl_bound} and Proposition~\ref{prop:21}, the one-sided trace operators $\gamma_d^{\pm}$ and 
$\gamma_n^{\pm}$ are well defined on $\mathscr H^1(\mathbb R^2\setminus\varGamma)$ and bounded.
The purpose of this section is to extend these operators to the space 
 $\mathscr H^0(\mathbb R^2\setminus\varGamma)$. Since the single and double layer potentials defined in Theorem~\ref{exten_layer}
are equal to the sum of a function in $\mathscr H^0(\mathbb R^2\setminus\varGamma)$ plus a classical single or double layer potential, 
we will be able to deduce at once traces results for surface potentials. As in Section~\ref{sec:polyn}, we continue assuming 
 that $\varGamma$ is a $\mathcal C^{1,1}$ curvilinear polygon. Recall that $\mathscr B_n(\mathbb R^2)=\big({W^2_d}(\mathbb R^2)\cap {W^2_n}(\mathbb R^2)\big)^\perp\cap {W^2_n}(\mathbb R^2)$
and $\mathscr B_d(\mathbb R^2)=\big({W^2_d}(\mathbb R^2)\cap {W^2_n}(\mathbb R^2)\big)^\perp\cap {W^2_d}(\mathbb R^2)$.
\par
To every $v\in \mathscr H^0(\mathbb R^2\setminus\varGamma)$, we associate  $\mathsf J_dv$ the element of ${\mathcal H}^{-1/2}_d(\varGamma)$ defined by:
\begin{equation}
\label{weak_trace0}
\llangle \mathsf J_dv,q\rrangle_{-\frac12,\frac12,d}=-\big(v,\Delta\mathcal L_d q\big)_{L^2(\mathbb R^2)}\forallt q\in
{\mathcal H}^{1/2}_d(\varGamma),
\end{equation}
where the operator $\mathcal L_d$ is given in \eqref{isom:D2}. We are going to show that $\mathsf J_dv$ is the ``jump'' of the 
one-sided Dirichlet 
traces of $v$ across $\varGamma$.
We denote by $\mathscr H_d^0(\mathbb R^2\setminus\varGamma)$ the image of the space 
$\mathscr B_d(\mathbb R^2)$ by the Laplacian. The operator:
$$\fct{\Delta_d : \mathscr B_d(\mathbb R^2)}
{\mathscr H_d^0(\mathbb R^2\setminus\varGamma)}{u}{\Delta u,}$$
being isometric,  $\mathscr H_d^0(\mathbb R^2\setminus\varGamma)$ is closed and we denote by $\Pi_d^0$ the orthogonal 
projection on this space in $L^2(\mathbb R^2)$.
%
%
It can be readily verify that: 
\begin{equation}
\label{hjuiqqw}
\mathsf J_dv=-\mathcal T_d\circ\mathcal L_d^{-1}\circ\Delta_d^{-1}\circ\Pi_d^0 v\forallt v\in \mathscr H^0(\mathbb R^2\setminus\varGamma),
\end{equation}
where the operator $\mathcal T_d$ is defined in \eqref{tD}.
Since the operators $\mathcal T_d$, $\mathcal L_d$ and $\Delta_d$ are isometric, it follows that:
\begin{equation}
\label{bound:1}
\|\mathsf J_dv\|_{-\frac12,d}=\|\Pi_d^0 v\|_{L^2(\mathbb R^2)}\leqslant \|v\|_{L^2(\mathbb R^2)}\forallt 
v\in \mathscr H^0(\mathbb R^2\setminus\varGamma).
\end{equation}
%
%
We turn now our attention to the Neumann trace operator. For every $v\in  \mathscr H^0(\mathbb R^2\setminus\varGamma)$, we denote by $\mathsf J_nv$ the element of ${\mathcal H}^{-3/2}_n(\varGamma)$ defined by:
\begin{equation}
\label{weak_trace1}
\llangle \mathsf J_nv,p\rrangle_{-\frac32,\frac32,n}=-\big(v,\Delta\mathcal L_n p\big)_{L^2(\mathbb R^2)}\forallt p\in
{\mathcal H}^{3/2}_n(\varGamma).
\end{equation}
%
When $p=\mathbf 1_\varGamma$ (which is indeed in $\mathcal H_n^{3/2}(\varGamma)$), $\mathcal L_n p=\mathbf 1_{\mathbb R^2}$ and 
therefore, the operator $\mathsf J_n$ is valued in:
\begin{equation}
\widetilde{\mathcal H}^{-3/2}_n(\varGamma)=\big\{q\in {\mathcal H}^{-3/2}_n(\varGamma)\,:\,\llangle q,\mathbf 1_\varGamma\rrangle_{-\frac32,\frac32,n}=0
\big\}.
\end{equation}
Looking for an expression like \eqref{hjuiqqw} for $\mathsf J_n$, we introduce the orthogonal decomposition 
$\mathscr B_n=\widetilde{\mathscr B}_n\oplus\langle \mathbf 1_{\mathbb R^2}\rangle$ of the space $\mathscr B_n$, where 
$\widetilde{\mathscr B}_n=\{u\in \mathscr B_n\,:\, \mu(\gamma_du)=0\}$. We introduce as well 
the space $\widetilde{\mathscr H}_n^0(\mathbb R^2\setminus\varGamma)=\{v\in  {\mathscr H}_n^0(\mathbb R^2\setminus\varGamma)\,:\,
(\mathbf 1_{\varOmega^-},v)_{L^2(\mathbb R^2)}=0\}$ and the isometric operator:%
$$\fct{\widetilde\Delta_n : \big(\widetilde{\mathscr B}_n,\|\cdot\|_D\big)}
{\widetilde{\mathscr H}_n^0(\mathbb R^2\setminus\varGamma)}{u}{\Delta u.}$$
Recalling that the operator $\mathcal T_{\!n}$ is defined in \eqref{tD} and denoting by $\widetilde{\Pi}_n^0$ the orthogonal projector 
onto $\widetilde{\mathscr H}_n^0(\mathbb R^2\setminus\varGamma)$ in $L^2(\mathbb R^2)$, we establish that:
$$\mathsf J_nv=-\mathcal T_{\!n}\circ\mathcal L_n^{-1}\circ\widetilde{\Delta}_n^{-1}\circ\widetilde{\Pi}_n^0 v
\forallt v\in \mathscr H^0(\mathbb R^2\setminus\varGamma).$$
We deduce that:
\begin{equation}
\label{bound:2}
\|\mathsf J_nv\|_{-\frac32,n}=\|\widetilde{\Pi}_n^0 v\|_{L^2(\mathbb R^2)}\leqslant \|v\|_{L^2(\mathbb R^2)}\forallt 
v\in \mathscr H^0(\mathbb R^2\setminus\varGamma).
\end{equation}
On the other hand, for every $v\in  \mathscr H^0(\mathbb R^2\setminus\varGamma)$ we can define $v^+$ and $v^-$ in $\mathscr H^0(\mathbb R^2\setminus\varGamma)$ by:
$$v^+=\begin{cases} 0&\text{ on }\varOmega^-\\
v|_{\varOmega^+}&\text{ on }\varOmega^+
\end{cases}\qquad\text{and}\qquad
v^-=\begin{cases} v|_{\varOmega^-}&\text{ on }\varOmega^-\\
0&\text{ on }\varOmega^+.
\end{cases}
$$
We can now  make precise the notion of trace for functions in $\mathscr H^0(\mathbb R^2\setminus\varGamma)$:
\begin{definition}
\label{def_trace_ext}
For every function $v\in  \mathscr H^0(\mathbb R^2\setminus\varGamma)$, we define the one-sided Dirichlet trace operators 
$\gamma_d^+$ and $\gamma_d^-$ by:
\begin{subequations}
\label{tr_ext}
\begin{equation}
\label{diri_ext}
\fct{\gamma_d^\pm:\mathscr H^0(\mathbb R^2\setminus\varGamma)}{{\mathcal H}^{-1/2}_d(\varGamma)}{v}{\pm \mathsf J_d{v^\pm}.}
\end{equation}
We define as well the one-sided Neumann trace operators  $\gamma_n^+$ and $\gamma_n^-$ by:
\begin{equation}
\label{neu_ext}
\fct{\gamma_n^\pm:\mathscr H^0(\mathbb R^2\setminus\varGamma)}{{\widetilde{\mathcal H}}^{-3/2}_n(\varGamma)}{v}{\mathsf J_n{v^\pm}.}
\end{equation}
\end{subequations}
\end{definition}
As expected, we have:
\begin{prop}
The operators \eqref{tr_ext} are bounded and are the extensions by density of the operators $\gamma_d^{\pm}$ and $\gamma_n^{\pm}$ 
(defined on $\mathscr H^1(\mathbb R^2\setminus\varGamma)$) to $\mathscr H^0(\mathbb R^2\setminus\varGamma)$. 
\end{prop}
\begin{proof}
The boundedness results from \eqref{bound:1} and \eqref{bound:2}.
Let $v$ be in $\mathscr H^1(\mathbb R^2\setminus\varGamma)$ and $q$ be in $\mathcal H^{1/2}_d(\varGamma)$. 
Green's formula \eqref{green} 
leads to:
$$\big(v^-,\Delta\mathcal L_d q\big)_{L^2(\mathbb R^2)}=\langle q,\gamma_d^-v\rangle_{-\frac12,\frac12}
=(q,\gamma_d^-v)_{L^2(\mathbb R^2)}=\llangle \gamma_d^-v,q\rrangle_{-\frac12,\frac12,d},$$
the last equality resulting from \eqref{second_chain} (the inclusions allowing $\gamma_d^-v$ to be considered as an element 
of $\mathcal H^{-1/2}_d(\varGamma)$ and asserting that the duality pairing 
$\llangle \cdot,\cdot\rrangle_{-\frac12,\frac12,d}$ extends the $L^2$ inner product).
On the other hand, according to \eqref{weak_trace0}:
$$\big(v^-,\Delta\mathcal L_d q\big)_{L^2(\mathbb R^2)}=-\llangle \mathsf J_d{v^-},q\rrangle_{-\frac12,\frac12,d}.$$
It follows that $\llangle \mathsf J_d{v^-}+\gamma_d^-v,q\rrangle_{-\frac12,\frac12,d}=0$ for every $q\in\mathcal H^{1/2}_d(\varGamma)$
and therefore $\mathsf J_d{v^-}=-\gamma_d^-v$ in $\mathcal H^{-1/2}_d(\varGamma)$. The proof of the other equalities ($\gamma_d^+v=
\mathsf J_d{v^+}$ and 
$\gamma^\pm_nv=\mathsf J_n{v^\pm}$)  follows from the same arguments.
\end{proof}
%
For every $u\in\mathscr H^0(\mathbb R^2\setminus\varGamma)$, we introduce the classical notations:
$$\big[\gamma_du\big]_\varGamma=\gamma_d^+u^+-\gamma_d^-u^-\qquad\text{and}\qquad
\big[\gamma_nu\big]_\varGamma=\gamma_n^+u^++\gamma_n^-u^-,$$
so that $\big[\gamma_du\big]_\varGamma=\mathsf J_dv$ in \eqref{weak_trace0} and $\big[\gamma_nu\big]_\varGamma=\mathsf J_nv$ in 
\eqref{weak_trace1}.
Before proving the jump relations for the single and the double layer potentials, we need to establish a preliminary technical result:
\begin{lemma}
\label{qqsiso}
\begin{enumerate}
\item
For every $p\in \mathcal H^{3/2}(\varGamma)$, $\big[\gamma_d \Delta  {\mathsf L}_d^S p\big]_\varGamma=0$ and for every 
$q\in\mathcal H^{1/2}(\varGamma)$,  $\big[\gamma_n \Delta  {\mathsf L}_n^D p\big]_\varGamma=0$.
\item Recall the the operator $\mathsf T_{\!d}$ and $\mathsf T_{\!n}$ are defined in \eqref{def_gelf_op} and the operators $\mathcal T_d$ and $\mathcal T_{\!n}$ in \eqref{tD}. The following identities hold:
\vspace{-2mm}
\begin{subequations}
\begin{alignat}{3}
\label{knklpoo}
-\big[\gamma_n\Delta {\mathsf L}_d^S p\big]_\varGamma+\sum_{k=1}^3\langle \mathsf q_k,p\rangle_{-\frac12,\frac12}
\mathsf q_k&=\mathsf T_{\! d}p&&\forallt p\in\mathcal H^{3/2}(\varGamma),\\
\label{knklpoogt}
-\big[\gamma_d\Delta \mathsf L_n^D q\big]_\varGamma
+\sum_{j=1}^2(\mathsf p_j,q)_{L^2(\varGamma)}\mathsf p_j&=\mathsf T_{\!n}q&&\forallt q\in\mathcal H^{1/2}(\varGamma),\\
\label{ghbnhk}
-\big[\gamma_n\Delta {\mathcal L}_np\big]_\varGamma
+\mu(p)|\varGamma|^{-1}\mathbf 1_\varGamma&=\mathcal T_{\!n} p&&\forallt p\in {\mathcal H}^{3/2}_n(\varGamma)\\
\label{vbwppl}
-\big[\gamma_d\Delta \mathcal L_dq\big]_\varGamma&=\mathcal T_d q&&\forallt q\in {\mathcal H}^{1/2}_d(\varGamma).
\end{alignat}
From \eqref{knklpoo} and \eqref{knklpoogt} we deduce in particular that the operators:
\begin{equation}
\label{45e}
\mathsf T_{\! d}:\mathscr A_S^{1/2}\longrightarrow \mathscr A_S^{-1/2}\qquad\text{and}\qquad
\mathsf T_{\! n}:\mathscr A_D^{-1/2}\longrightarrow \mathscr A_D^{1/2}
\end{equation}
are isometric. 
\end{subequations}
%
\end{enumerate}
\end{lemma}
\begin{proof}
The first assertion of the Lemma results from the combination of \eqref{op:L} and  \eqref{op:calL} (that make precise the ranges of the operators ${\mathsf L}_d^S$, 
${\mathsf L}_n^D$, $\mathcal L_d$ and $\mathcal L_n$) and  the definitions \eqref{weak_trace0} and \eqref{weak_trace1} of 
the jumps of  the Dirichlet and Neumann one-sided traces.
For the second assertion, let us verify that for every $p\in{\mathcal H}^{3/2}(\varGamma)$:
$$
-\big[\gamma_n\Delta {\mathsf L}_d^S p\big]_\varGamma+\sum_{k=1}^3\langle \mathsf q_k,p\rangle_{-\frac12,\frac12}
\mathsf q_k=\mathsf T_{\! d}p,
$$
where $\mathsf T_{\! d}$ is the isometric operator defined in \eqref{def_gelf_op}.
Thus, we have: 
$$\llangle \mathsf T_{\! d}p,\gamma_d\theta\rrangle_{-\frac32,\frac32}=\big({\mathsf L}_d^Sp, {\mathsf L}_d^S\circ \gamma_d\theta\big)_S
\forallt \theta\in W^2(\mathbb R^2).$$
On the other hand, according to \eqref{nice:rel}:
$$( {\mathsf L}_d^Sp, {\mathsf L}_d^S\circ \gamma_d\theta)_S=( {\mathsf L}_d^Sp,\Pi^S_d\theta)_S=( {\mathsf L}_d^Sp,\theta)_S.$$
Choosing $\theta={\mathcal L}_n\tilde p$ with $\tilde p$ any element in the space ${\mathcal H}^{3/2}_n(\varGamma)$, we get:
$$\llangle \mathsf T_{\! d}p,\tilde p\rrangle_{-\frac32,\frac32}=({\mathsf L}_d^Sp,{\mathcal L}_n\tilde p)_S=(\Delta {\mathsf L}_d^Sp,
\Delta {\mathcal L}_n\tilde p)_{L^2(\mathbb R^2)}+\sum_{k=1}^3\langle \mathsf q_k,p\rangle_{-\frac12,\frac12}
\langle \mathsf q_k,\tilde p\rangle_{-\frac12,\frac12},$$
which, once compared with \eqref{weak_trace1}, yields the result. The proof of the other equalities are similar.
\par
Regarding \eqref{45e}, it suffices to notice that $\mathsf L_d^Sp$ is an affine function when $p\in\mathscr A_S^{1/2}$ and the same 
observation applies to $\mathsf L_n^Dq$ if $q$ belongs to $\mathscr A_D^{-1/2}$. We conclude with 
equalities \eqref{knklpoo} and \eqref{knklpoogt}.
\end{proof}
With the definition of the trace operators given in Definition~\ref{def_trace_ext}:
\begin{prop}[Jump relations]
\label{jps}
The following equalities hold:
\begin{subequations}
\begin{alignat}{3}
\label{eq_jump_1}
\gamma_d^+\circ \mathscr S^\dagger_\varGamma-\gamma_d^-\circ \mathscr S^\dagger_\varGamma&=0&
\gamma_n^+\circ \mathscr D^\dagger_\varGamma+\gamma_n^-\circ \mathscr D_\varGamma^\dagger&=0\\
\label{eq_jump_2}
\gamma_n^+\circ \mathscr S^\dagger_\varGamma+\gamma_n^-\circ \mathscr S^\dagger_\varGamma&={\rm Id}&\qquad
\gamma_d^+\circ \mathscr D^\dagger_\varGamma-\gamma_d^-\circ \mathscr D^\dagger_\varGamma&={\rm Id}.
\end{alignat}
\end{subequations}
\end{prop}
Notice that the operators $\gamma_n^\pm\circ \mathscr S^\dagger_\varGamma$ map $\mathcal H^{-3/2}(\varGamma)$ 
into the larger space $\mathcal H^{-3/2}_n(\varGamma)$. The first relation in \eqref{eq_jump_2}
means that there is some sort of compensation which makes the sum of 
both terms $\gamma_n^\pm\circ \mathscr S^\dagger_\varGamma$ more regular than each one taken separately. The same remark holds for 
the second identity in  \eqref{eq_jump_2} with the spaces $\mathcal H^{-1/2}(\varGamma)$ and $\mathcal H^{-1/2}_d(\varGamma)$. This  contrasts with 
what happens when the domain is of class $\mathcal C^{1,1}$ where
$\mathcal H^{-3/2}(\varGamma)=\mathcal H^{-3/2}_n(\varGamma)=H^{-3/2}(\varGamma)$ and $\mathcal H^{-1/2}(\varGamma)=\mathcal H^{-1/2}_d(\varGamma)=H^{-1/2}(\varGamma)$.
\begin{proof}Let $q$ be in $\mathcal H^{-3/2}(\varGamma)$. According to \eqref{bounded_SD1}:
$$\mathscr S_\varGamma^\dagger q=-\Delta {\mathsf L}_d^S\circ \mathsf T_{\! d}^{-1}q+\sum_{j=1}^3
\langle \mathsf q_j,\mathsf T_{\! d}^{-1}q\rangle_{-\frac12,\frac12}\mathscr S_\varGamma \mathsf q_j.$$
For every $\tilde q\in \mathcal H^{1/2}_d(\varGamma)$:
$$\big(\Delta {\mathsf L}_d^S\circ \mathsf T_{\! d}^{-1}q,\Delta \mathcal L_d \tilde q\big)_{L^2(\mathbb R^2)}=
\big(  {\mathsf L}_d^S\circ \mathsf T_{\! d}^{-1}q, \mathcal L_d \tilde q\big)_S=0,$$
because ${\mathsf L}_d^S\circ \mathsf T_{\! d}^{-1}q\in W^2_d(\mathbb R^2)^\perp$ (see \eqref{isom:FRT}) and 
$\mathcal L_d \tilde q\in W^2_d(\mathbb R^2)$ (see \eqref{isom:D2}). According to \eqref{weak_trace0}, we deduce that 
$$\mathsf J_d \big(\Delta {\mathsf L}_d^S\circ \mathsf T_{\! d}^{-1}q\big)= 
\big[\gamma_d \Delta {\mathsf L}_d^S\circ \mathsf T_{\! d}^{-1}q\big]_\varGamma=0,$$
and since $\big[\gamma_d\mathscr S_\varGamma \mathsf q_j\big]_\varGamma=0$ for $j=1,2,3$, we have proved the first equality in 
\eqref{eq_jump_1}. Continuing with the single layer potential, the first equality in \eqref{eq_jump_2} is a straightforward consequence of
\eqref{knklpoo}. The proofs of the relations related to 
 the double layer potential are  similar.
\end{proof}
The rest of this section is devoted to establishing additional properties 
concerning the traces of harmonic functions in $L^2_{\ell oc}(\mathbb R^2)$. To do this, we must first establish a technical lemma.
\begin{lemma}
\label{lem_decomp}
For every $u\in \mathscr H^0(\mathbb R^2\setminus\varGamma)$, there exist 
\begin{itemize}
\item[--] $p_1\in\mathcal H^{3/2}(\varGamma)$ and $q_1\in\mathcal H_d^{1/2}$ such that $u=\Delta\mathsf L_d^Sp_1+\Delta\mathcal L_dq_1$;
\item[--] $p_2\in\mathcal H^{3/2}_d(\varGamma)$ and $q_2\in\mathcal H^{1/2}$ such that $u=\Delta\mathcal L_n p_2+\Delta\mathsf L_n^Dq_2$.
\end{itemize}
\end{lemma}
\begin{proof}
Let $u$ be in  $\mathscr H^0(\mathbb R^2\setminus\varGamma)$. According to \eqref{def_harmL2}, there exists 
$v\in \big({W^2_d}(\mathbb R^2)\cap {W^2_n}(\mathbb R^2)\big)^\perp$ such that $u=\Delta v$. In line with the orthogonal decomposition
in $\big(W^2(\mathbb R^2);\|\cdot\|_S\big)$:
$$\big({W^2_d}(\mathbb R^2)\cap {W^2_n}(\mathbb R^2)\big)^\perp=
\underbrace{\big[\big({W^2_d}(\mathbb R^2)\cap {W^2_n}(\mathbb R^2)\big)^\perp\cap W^2_d(\mathbb R^2)\big]}_{\mathscr B_d(\mathbb R^2)}\oplus W^2_d(\mathbb R^2)^\perp,$$
we can decompose $v$ as $v=v_1+v_2$ and there exists $q_1\in \mathcal H_d^{1/2}(\varGamma)$ such that $v_1=\mathcal L_dq_1$ 
(see \eqref{isom:D2}) and $p_1\in \mathcal H^{3/2}(\varGamma)$ such that $v_2=\mathsf L_d^S p_1$ (see \eqref{isom:FRT}). This proves 
the first point of the Proposition. For the second, we use the orthogonal decomposition in $\big(W^2(\mathbb R^2);\|\cdot\|_D\big)$:
$$\big({W^2_d}(\mathbb R^2)\cap {W^2_n}(\mathbb R^2)\big)^\perp=\underbrace{
\big[\big({W^2_d}(\mathbb R^2)\cap {W^2_n}(\mathbb R^2)\big)^\perp\cap W^2_n(\mathbb R^2)\big]}_{\mathscr B_n(\mathbb R^2)}\oplus W^2_n(\mathbb R^2)^\perp,$$
and we conclude the same way.
\end{proof}
\begin{theorem}
\label{mpsdt}
Let $u$ be in $\mathscr H^0(\mathbb R^2\setminus\varGamma)$. 
\begin{enumerate}
\item If $[\gamma_d u]_\varGamma\in {\mathcal H}^{-1/2}(\varGamma)$ 
or $[\gamma_n u]_\varGamma\in \mathcal H^{-3/2}(\varGamma)$ then $\big([\gamma_d u]_\varGamma,[\gamma_n u]_\varGamma\big)
\in{\mathcal H}^{-1/2}(\varGamma)\times   \mathcal H^{-3/2}(\varGamma)$.
\item If $[\gamma_d v]_\varGamma=0$, then $v= \mathscr S_\varGamma^\dagger [\gamma_n v]_\varGamma$. 
If $[\gamma_n v]_\varGamma=0$, then $v= \mathscr D_\varGamma^\dagger  [\gamma_d v]_\varGamma$.
\item \label{3:gggh} If $[\gamma_d u]_\varGamma=0$ and $[\gamma_n u]_\varGamma=0$ then $u=0$.
\end{enumerate}
\end{theorem}
The first point seems particularly noteworthy. It means that as soon as the jump of the one-sided Dirichlet traces or the jump of the one-sided Neumann traces 
is ``regular'' (in full generality $[\gamma_d u]_\varGamma$ is only in $\mathcal H^{-1/2}_d(\varGamma)$ 
and $[\gamma_n u]_\varGamma$  in $\mathcal H^{-3/2}_n(\varGamma)$), the other jump inherits the same regularity.
\begin{proof}Addressing the first point of the Theorem, let
assume that $[\gamma_d u]_\varGamma\in {\mathcal H}^{-1/2}(\varGamma)$. According to \eqref{expand_asymp} 
(the asymptotic expansions of the single layer potential and of harmonic functions) there exists $\mathsf q\in\mathscr A_S^{-1/2}$ such 
that $\mathscr D^\dagger_\varGamma[\gamma_d u]_\varGamma-\mathscr S_\varGamma\mathsf q\in L^2(\mathbb R^2)$. 
Let $v=u-\mathscr D^\dagger_\varGamma[\gamma_d u]_\varGamma+\mathscr S_\varGamma\mathsf q$. This function is in 
$\mathscr H^0(\mathbb R^2\setminus\varGamma)$ and satisfies $[\gamma_d v]_\varGamma=0$ and $[\gamma_n v]_\varGamma=
[\gamma_n u]_\varGamma+\mathsf q$. This entails that for our purpose, 
up to replacing $u$ by 
$v$, 
we can assume that $[\gamma_d u]_\varGamma=0$. According to Lemma~\ref{lem_decomp}, the function $u$ can be decomposed 
as $u=\Delta\mathsf L_d^Sp_1+\Delta\mathcal L_dq_1$ with $p_1\in\mathcal H^{3/2}(\varGamma)$ and $q_1\in\mathcal H_d^{1/2}$. 
Invoking next the first point of Lemma~\ref{qqsiso}, we obtain:
$$\big[\gamma_du\big]_\varGamma=\big[\gamma_d\Delta\mathsf L_d^Sp_1\big]_\varGamma
+\big[\gamma_d\Delta\mathcal L_dq_1\big]_\varGamma=\big[\gamma_d\Delta\mathcal L_dq_1\big]_\varGamma=0,$$
which entails that $q_1=0$ with \eqref{vbwppl}. It follows that $\big[\gamma_nu\big]_\varGamma=\big[\gamma_n\Delta\mathsf L_d^Sp_1\big]_\varGamma$ and 
therefore $\big[\gamma_nu\big]_\varGamma\in \mathcal H^{-3/2}(\varGamma)$ according to \eqref{knklpoo}.
\par
In a similar fashion, assuming that $[\gamma_n u]_\varGamma\in   \mathcal H^{-3/2}(\varGamma)$ can be reduced to assuming that 
$[\gamma_n u]_\varGamma=0$ up to replacing $u$ by $u-\mathscr S^\dagger_\varGamma[\gamma_n u]_\varGamma+\mathscr D_\varGamma \mathsf p$ for some $\mathsf p\in\mathscr A_D^{1/2}$. Then 
we use the latter decomposition provided by Lemma~\ref{lem_decomp}:
$$u=\Delta\mathcal L_n p_2+\Delta\mathsf L_n^Dq_2,$$
for some $p_2\in\mathcal H^{3/2}_d(\varGamma)$ and $q_2\in\mathcal H^{1/2}$. Based on Lemma~\ref{qqsiso}, we deduce that:
$$[\gamma_nu]_\varGamma=\big[\gamma_n\Delta\mathcal L_n p_2\big]_\varGamma+
\big[\gamma_n\Delta\mathsf L_n^Dq_2\big]_\varGamma
=\big[\gamma_n\Delta\mathcal L_n p_2\big]_\varGamma=0.$$
This condition means, according to \eqref{ghbnhk}, that $\mathcal T_{\!n} p_2=\mu(p_2)|\varGamma|^{-1}\mathbf 1_\varGamma$ where the operator 
$\mathcal T_{\!n}$ is defined in  \eqref{tD}. We deduce that $p_2=\mu(p_2)\mathbf 1_\varGamma$ and then that $\mathcal L_n p_2=\mu(p_2)
\mathbf 1_{\mathbb R^2}$. Finally, $\Delta \mathcal L_n p_2=0$ and $[\gamma_du]_\varGamma=
\big[\gamma_d\Delta\mathsf L_n^Dq_2\big]_\varGamma$ which belongs to $\mathcal H^{-1/2}(\varGamma)$ according to 
\eqref{knklpoogt}. 
\par
We consider now the second assertion of the Theorem. If $[\gamma_d v]_\varGamma=0$ then $[\gamma_n v]_\varGamma$ belongs to $\mathcal H^{-3/2}(\varGamma)$ according to the first point of the Theorem.
 Let $\mathsf q$ be in $\mathscr A_S^{-1/2}$ such that $\mathscr S_\varGamma^\dagger ([\gamma_n v]_\varGamma-\mathsf q)$
is in $\mathscr H^0(\mathbb R^2\setminus\varGamma)$ and introduce 
$u=v-\mathscr S_\varGamma^\dagger ([\gamma_n v]_\varGamma-\mathsf q)$.
 This function is in $\mathscr H^0(\mathbb R^2\setminus\varGamma)$ and satisfies $[\gamma_d u]_\varGamma=0$ and 
 $[\gamma_n u]_\varGamma=\mathsf q$. Proceeding as in the proof of the first point of the theorem, this entails that $u=\Delta\mathsf L_d^S p$ for some $p\in\mathcal H^{3/2}(\varGamma)$, with $[\gamma_n\Delta\mathsf L_d^Sp]_\varGamma=\mathsf q$. This means, with  \eqref{knklpoo}
 that $$\mathsf T_{\! d}p=\sum_{k=1}^3\langle \mathsf q_k,p\rangle_{-\frac12,\frac12}
\mathsf q_k-\mathsf q,$$
and therefore that $p\in\mathscr A_S^{1/2}$ taking into account \eqref{45e}. It follows that $\mathsf L_{\!d}^Sp\in\mathscr A$ and then $u=0$,
 $\mathsf q=0$ and finally $v=
 \mathscr S_\varGamma^\dagger [\gamma_n v]_\varGamma$. We proceed in the same manner to prove the other statements involving 
 the double layer potential and since the last point of the theorem is obvious, the proof is  complete.
\end{proof}
%
\section{The Laplace equation in $L^2$}
\label{sec:repres}
In this section, we assume that $\varGamma$ is a straight polygon.
An important point to keep in mind when looking for solutions in $L^2_{\ell oc}$ to Dirichlet and Neumann problems in a polygonal domain, 
is the loss of uniqueness. Indeed, there exist non-zero harmonic functions in $L^2(\varOmega^+)$ and in $L^2(\varOmega^-)$ 
with zeros Dirichlet data or with zero Neumann data.
\begin{figure}[h]
\centerline{\input{simple_coucge_2.tex}}
\caption{\label{Fig1} Examples of domains for which there exist square integrable harmonic functions with vanishing boundary data.}
\end{figure}
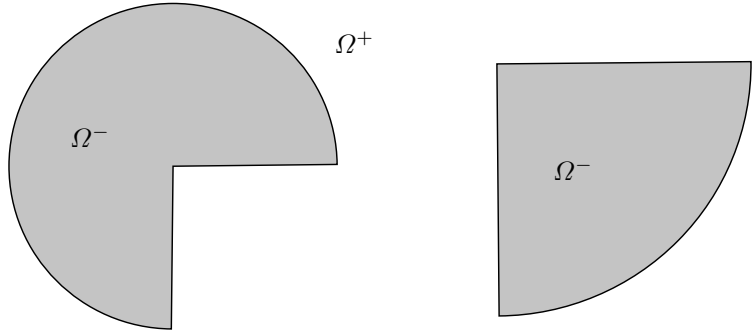
As explained in \cite{Dauge:1987tg}, in the domain $\varOmega^-$ on the left of Fig.~\ref{Fig1}, there exists a square integrable harmonic function 
with zero Dirichlet 
data. In polar coordinates define the function $U(r,\theta)=r^{-2/3}\sin(2\theta/3)$ and let $\eta\in\mathscr D(\mathbb R^2)$ be a cut-off function equal 
to 1 near the corner. Then let $X$ be the variationnal solution in $H^1_0(\varOmega^-)$ to the problem $\Delta X=\Delta(\eta U)$ 
(notice that the right hand side is smooth in $\varOmega^-$). 
The function $\eta U-X$ is in $L^2(\varOmega^-)$, non zero (because $X$ belongs to $H^1(\varOmega^-)$ and $\eta U$ does not), harmonic in $\varOmega^-$ and equal to zero on the boundary of the domain. There exists also a harmonic function 
with zero Neumann data which is equal, near the corner to $r^{-2/3}\cos(2\theta/3)$ (this is actually 
the harmonic conjugate of the preceeding one). The same constructions apply with the domain on the right of 
Fig.~\ref{Fig1} and provide examples of non-zero harmonic functions in $\varOmega^+$ with zero Dirichlet or Neumann data.
%
\begin{theorem}[Solvability of the interior Dirichlet problem]
\label{diri_interior}
For every $p\in{\mathcal H}^{-1/2}_d(\varGamma)$, there exists $v_p$ in $\mathscr H^0(\varOmega^-)$ such that $\gamma_d^-v_p=p$. Moreover, the application $p\longmapsto v_p$ is continuous from $\mathcal H^{-1/2}_d(\varGamma)$ to $\mathscr H^0(\varOmega^-)$.
\end{theorem}
\begin{proof}
Recall that, for every $q\in\mathcal H_d^{1/2}(\varGamma)$:
$$\|q\|_{\frac12,d}=\inf\big\{\|u\|_S\,:\,u\in {W^2_d}(\mathbb R^2),\,\gamma_nu=q\big\}=\|\mathcal L_d q\|_S.$$
According to \cite[Theorem 10.4.1]{Agranovich:2015us}, there exists 
an extension operator from $H^2(\varOmega^-)$ to $H^2(\mathbb R^2)$ and since $H^2(\mathbb R^2)$ is continuously embedded in 
$W^2(\mathbb R^2)$, we can assume that this operator is valued in this latter space. This yields the existence of a constant $C>0$ 
such that:
\begin{subequations}
\label{mmmpppmm}
\begin{equation}
\|q\|_{\frac12,d}
\leqslant C\|\mathcal L_d q|_{\varOmega^-}\|_{H^2(\varOmega^-)}\forallt q\in\mathcal H^{1/2}_d(\varGamma).
\end{equation}
On the other hand, according to \cite[Theorem 4.3.1.4]{Grisvard:1985aa} and since $\varOmega^-$ is assumed 
to be a straight polygonal domain, there exists a constant $C>0$ such that:
\begin{equation}
\|\mathcal L_dq|_{\varOmega^-}\|_{H^2(\varOmega^-)}\leqslant C \big(\|\Delta \mathcal L_d q\|_{L^2(\varOmega^-)}+
\|\mathcal L_d q\|_{L^2(\varOmega^-)}\big)\forallt q\in{\mathcal H}^{1/2}_d(\varGamma).
\end{equation}
\end{subequations}
We deduce from both estimates \eqref{mmmpppmm} that, on the space $\mathcal H^{1/2}_d(\varGamma)$, the norm deriving from the scalar product:
$$\big(\Delta\mathcal L_d q_1,\Delta\mathcal L_d q_2\big)_{L^2(\varOmega^-)}+
\big(\mathcal L_d q_1,\mathcal L_d q_2\big)_{L^2(\varOmega^-)}\forallt q_1,q_2\in{\mathcal H}^{1/2}_d(\varGamma),$$
is equivalent to the norm $\|\cdot\|_{\frac12,d}$ associated to the scalar product \eqref{def_norm_10}.  Riesz representation theorem 
asserts that for 
every $p\in {\mathcal H}^{-1/2}_d(\varGamma)$, there exists $q_p\in {\mathcal H}^{1/2}_d(\varGamma)$ such that:
$$\llangle p,q\rrangle_{-\frac12,\frac12,d}=\big(\Delta\mathcal L_d q_p,\Delta\mathcal L_d q\big)_{L^2(\varOmega^-)}+\big(\mathcal L_d q_p,\mathcal L_d q\big)_{L^2(\varOmega^-)}\forallt q\in
{\mathcal H}^{1/2}_d(\varGamma),$$
and that  the mapping $p\longmapsto q_p$ is continuous from 
$\mathcal H^{-1/2}_d(\varGamma)$ to $\mathcal H^{1/2}_d(\varGamma)$.
Let now $w_p$ be the unique solution in $H^1_0(\varOmega^-)$ of the variational Dirichlet problem:
$$(\nabla w_p,\nabla\theta)_{L^2(\varOmega^-;\mathbb R^2)}=-\big(\mathcal L_d q_p,\theta\big)_{L^2(\varOmega^-)}\forallt \theta\in H^1_0(\varOmega^-).$$
Since $p\longmapsto \mathcal L_d q_p|_{\varOmega^-}$ is continuous from $\mathcal H^{-1/2}_d(\varGamma)$ to $H^2(\varOmega^-)$, we 
deduce that $p\longmapsto w_p$ is also continuous from $\mathcal H^{-1/2}_d(\varGamma)$ to $H^1_0(\varOmega^-)$.
Applying   Green's formula \eqref{green}, we then obtain that:
$$\big(\mathcal L_d q_p,\mathcal L_d q\big)_{L^2(\varOmega^-)}=\big(w_p,\Delta \mathcal L_d q\big)_{L^2(\varOmega^-)}
\forallt q\in
{\mathcal H}^{1/2}_d(\varGamma).$$
Let $\Pi_{\varOmega^-}:L^2(\varOmega^-)\longrightarrow \mathscr H^0(\varOmega^-)$ be the Bergman projection i.e. the orthogonal projection 
in $L^2(\varOmega^-)$ onto the closed subspace $\mathscr H^0(\varOmega^-)$ of the harmonic functions and define the function $v_p= 
\Delta\mathcal L_d q_p+\Pi_{\varOmega^-}w_p$. It is clear that the mapping
$p\longmapsto v_p$ is continuous from $\mathcal H^{-1/2}_d(\varGamma)$ to $\mathscr H^0(\varOmega^-)$ and since:
$$\llangle p,q\rrangle_{-\frac12,\frac12,d}=\big(v_p,\Delta\mathcal L_d q\big)_{L^2(\varOmega^-)}
\forallt q\in
{\mathcal H}^{1/2}_d(\varGamma),$$
the proof is completed.
 \end{proof}
 %
%
Recall that $\mathscr A_S^{\frac12}$ is the three dimensional subspace of $H^{1/2}(\varGamma)$ spanned by the traces of the affine functions. 
Since $H^{1/2}(\varGamma)\subset L^2(\varGamma)\subset \mathcal H^{-1/2}_d(\varGamma)$, the space $\mathscr A_S^{\frac12}$ can 
also be seen as a subspace of $\mathcal H^{-1/2}_d(\varGamma)$ and we denote by $\mathscr A^{1/2}_{S,d}$ the space 
$\mathcal T_d^{-1}\mathscr A_S^{\frac12}$, where the operator $\mathcal T_d$ is defined in \eqref{tD}.  Therefore $\mathscr A^{1/2}_{S,d}$ is the subspace of $\mathcal H^{1/2}_d(\varGamma)$ such that $[\gamma_d\Delta \mathcal L_dq]_\varGamma\in \mathscr A_S^{1/2}$
for every $q\in \mathscr A^{1/2}_{S,d}$. It follows that 
$$\big(\mathscr A^{1/2}_{S,d}\big)^\perp=\big\{q\in\mathcal H^{1/2}_d(\varGamma)\,:\,(q,\theta)_{L^2(\varGamma)}=0\quad\text{for all }\,\theta\in\mathscr A_S^{1/2}
\big\}.$$
The spaces  $\mathscr A^{1/2}_{S,d}$ and $\big(\mathscr A^{1/2}_{S,d}\big)^\perp$ will be used in the proof of the next theorem.
\begin{theorem}[Solvability of the exterior Dirichlet problem]
\label{diri_exterior}
For every $p\in {\mathcal H}^{-1/2}_d(\varGamma)$, there exists $v_p$ in $\mathscr H^0(\varOmega^+)$ and
$\mathsf q_p\in \mathscr A^{-1/2}_S$ such that $\gamma_d^+(v_p+\mathscr S_\varGamma \mathsf q_p)=p$. 
\end{theorem}
The continuity of the solution with respect to the boundary data is not  clear so far. 
The proof relies on the following Lemma:
\begin{lemma}
\label{lmmme}
Let $\xi$ be a distribution in $H^{-1}(\varOmega^+)$ compactly supported in $\overline{\varOmega^+}$. Then there exists 
$p\in \mathscr A_S^{\frac12}$  such that the Dirichlet problem
\begin{equation}
\label{diri}
-\Delta u=\xi\quad\text{ in }\varOmega^+\qquad\text{and}\qquad \gamma_d^+u=p\quad\text{ on }\varGamma,
\end{equation}
admits a  solution in $H^1(\varOmega^+)$. If $\xi\in L^2(\varOmega^+)$, this solution is in $H^1(\varOmega^+,\Delta)$.
\end{lemma}
\begin{proof}
Following the method described in 
\cite{Amrouche:1997aa}, we 
introduce the weighted Sobolev space (remind that the functions $\rho$ and $\lg$ were defined earlier in \eqref{weighted}):
$$W^1_d(\varOmega^+)=\Big\{u\in \mathscr D'(\varOmega^+)\,:\, \frac{u}{\rho\,{\rm lg}}\in L^2(\varOmega^+),\,\frac{\partial u}{\partial x_j}\in L^2(\varOmega^+),\,(j=1,2)\text{ and }\gamma_d^+u=0\Big\}.$$
This space is strictly bigger than $H^1_0(\varOmega^+)$ and its purpose is that the norm deriving from the scalar product:
$$(\nabla u_1,\nabla u_2)_{L^2(\varOmega^+;\mathbb R^2)}\forallt u_1,u_2\in W^1_d(\varOmega^+),$$
is equivalent to the natural norm (i.e. Poincar\'e inequality holds true). Furthermore, the space $\mathscr D(\varOmega^+)$ is dense in $W^1_d(\varOmega^+)$ so that 
$W^1_d(\varOmega^+)$ is a distribution space. Let 
$v$ be 
the solution in  $W^1_d(\varOmega^+)$ of the variational problem:
\begin{equation}
\label{folpmpoll}
(\nabla v,\nabla \theta)_{L^2(\varOmega^+;\mathbb R^2)}=\langle \xi,\theta\rangle_{H^{-1}(\varOmega^+),
H^1(\varOmega^+)}\forallt \theta\in W^1_d(\varOmega^+).
\end{equation}
Since 
$W^1_d(\varOmega^+)\subset H^1_{\ell oc}(\varOmega^+)$ the term in the right hand side makes sens, recalling that $\xi$ is assumed 
to be compactly 
supported. The function $v$ being harmonic outside a compact set, according to \eqref{exp_dfolp}, it can be expanded in this region as:
$$v(x)=\sum_{j=0}^{+\infty}\mathfrak p_m(x)+\mathfrak q_0\ln|x|+\sum_{j=0}^{+\infty}\frac{\mathfrak q_m(x)}{|x|^{2m}},$$
where $\mathfrak q_0\in\mathbb R$ and, for every integer $m$, $\mathfrak p_m, \mathfrak q_m$ are harmonic polynomials on $\mathbb R^2$ of degree $m$. By definition of $W^1_d(\varOmega^+)$, 
the function $v/(\rho\lg)$ is in $L^2(\varOmega^+)$, which entails that $\mathfrak p_m=0$ for every $m\geqslant 1$ and $\mathfrak q_0=0$. On the other hand, according to \eqref{asymp_sl}, 
there exists $\mathsf q\in \mathscr A_S^{-1/2}(\varGamma)$ (the subspace of $H^{-1/2}(\varGamma)$ spanned by $\mathsf q_1,\mathsf q_2$ and 
$\mathsf q_2$) such that, for $|x|$ large:
$$\mathscr S_\varGamma \mathsf q(x)=\frac{\mathfrak q_1(x)}{|x|^2}+\mathscr O(1/|x|^2).$$
The function $u=v-\mathscr S_\varGamma \mathsf q|_{\varOmega^+}-\mathfrak p_0\mathbf 1_{\varOmega^+}$ is the solution we are looking for (in particular 
it is in $L^2(\varOmega^+)$). 

\end{proof}
We can now carry out the 
%
\begin{proof}[Proof of Theorem~\ref{diri_exterior}]
Let $D^-$ and $D^+$ be two large open disks containing $\varOmega^-$ such that $\overline{\varOmega^-}\subset D^-$ and 
$\overline{D^-}\subset D^+$. Let $\chi$ be a smooth cut-off function defined in $\mathbb R^2$ such that 
$0\leqslant \chi\leqslant 1$, $\chi=1$ in $D^-$ and $\chi=0$ in $\mathbb R^2\setminus\overline{D^+}$. Following the lines of the proof of Theorem~\ref{diri_interior}, 
the norm deriving from the scalar product:
$$
\big(\Delta\chi\mathcal L_d q_1,\Delta\chi\mathcal L_d q_2\big)_{L^2(\varOmega^+)}+
\big(\chi\mathcal L_d q_1,\chi\mathcal L_d q_2\big)_{L^2(\varOmega^+)}\forallt q_1,q_2\in{\mathcal H}^{1/2}_d(\varGamma),
$$
is equivalent to the norm    $\|\cdot\|_{\frac12,d}$ associated to the scalar product \eqref{def_norm_10} in the space 
$\mathcal H^{1/2}_d(\varGamma)$. Applying Riesz representation Theorem 
we deduce that for 
every $p\in {\mathcal H}^{-1/2}_d(\varGamma)$, there exists $q_p\in {\mathcal H}^{1/2}_d(\varGamma)$ such that:
\begin{equation}
\label{ghnjko}
\llangle p,q\rrangle_{-\frac12,\frac12,d}=\big(\Delta\chi\mathcal L_d q_p,\Delta\chi\mathcal L_d q\big)_{L^2(\varOmega^+)}+\big(\chi\mathcal L_d q_p,\chi\mathcal L_d q\big)_{L^2(\varOmega^+)}\forallt q\in
{\mathcal H}^{1/2}_d(\varGamma).
\end{equation}
The first term in the right hand side is next expanded as follows:
\begin{multline}
\label{mulkjiopp}
\big(\Delta\chi\mathcal L_d q_p,\Delta\chi\mathcal L_d q\big)_{L^2(\varOmega^+)}=\big(\chi\Delta\chi\mathcal L_d q_p,\Delta \mathcal L_d q\big)_{L^2(\varOmega^+)}+2\big(\Delta \chi\mathcal L_d q_p\nabla\chi,\nabla\mathcal L_dq\big)_{L^2(\varOmega^+,\mathbb R^2)}
\\+\big(\Delta \chi\Delta\chi\mathcal L_dq_p,\mathcal L_dq\big)_{L^2(\varOmega^+)}.
\end{multline}
Focusing on the second term of this expansion, Lemma~\ref{lmmme} ensures the existence of a function 
$ u_1\in H^1(\mathbb \varOmega^+)$ such that $\gamma_d^+ u_1\in\mathscr A_S^{1/2}$ and:
$$(\nabla u_1,\nabla \theta)_{L^2(\varOmega^+,\mathbb R^2)}=
\big(\Delta \chi\mathcal L_d q_p\nabla\chi,\nabla\theta\big)_{L^2(\varOmega^+,\mathbb R^2)}\forallt \theta\in W^1_d(\varOmega^+).$$
For every $k\geqslant 1$ denote by $\mathcal L_d^k q$ the function $\phi_k\mathcal L_dq$ where $\phi_k$ is the truncation function 
mentioned in Proposition~\ref{mqwxf}. For $k$ large enough, since the function $\Delta \chi\mathcal L_d q_p\nabla\chi$ 
is compactly supported, we are allowed to write:
\begin{align*}
\big(\Delta \chi\mathcal L_d q_p\nabla\chi,\nabla\mathcal L_d q\big)_{L^2(\varOmega^+,\mathbb R^2)}
&=\big(\nabla u_1,\nabla \mathcal L_d^k q\big)_{L^2(\varOmega^+,\mathbb R^2)}\\
&=-\big(\gamma_d^+ u_1,q\big)_{L^2(\varGamma)}-\big( u_1,\Delta  \mathcal L_d^k q\big)_{L^2(\varOmega^+)}.
\end{align*}
Assume from now on that $q$ is in $\big(\mathscr A^{1/2}_{S,d}\big)^\perp$.
If follows that $(\gamma_d^+ u_1,q)_{L^2(\varGamma)}=0$ and letting $k$ go  to $+\infty$ we obtain:
\begin{subequations}
\label{cvwqso}
\begin{equation}
\big(\Delta \chi\mathcal L_d q_p\nabla\chi,\nabla\mathcal L_d q\big)_{L^2(\varOmega^+,\mathbb R^2)}
=-\big( u_1,\Delta  \mathcal L_d q\big)_{L^2(\varOmega^+)}.
\end{equation}
Considering now the last term in \eqref{mulkjiopp}, we denote by $ u_2$ the function in $H^1(\varOmega^+)$, provided by Lemma~\ref{lmmme}, 
satisfying $-\Delta  u_2=\Delta \chi\Delta\chi\mathcal L_dq_p$ and $\gamma_d^+ u_2\in\mathscr A^{\frac12}_S$. For $k$ large enough, we 
can write that:
\begin{align*}
\big(\Delta \chi\Delta\chi\mathcal L_dq_p,\mathcal L_dq\big)_{L^2(\varOmega^+)}&=-\big(\Delta u_2,\mathcal L^k_dq\big)_{L^2(\varOmega^+)}\\
&=-\big(q,\gamma_d^+ u_2\big)_{L^2(\varGamma)}-\big( u_2,\Delta \mathcal L^k_dq\big)_{L^2(\varOmega^+)}.
\end{align*}
Again, since $q$ is assumed to be in $\big(\mathscr A^{1/2}_{S,d}\big)^\perp$ the boundary integral vanishes and letting $k$ go  to $+\infty$ we are left with:
\begin{equation}
\big(\Delta \chi\Delta\chi\mathcal L_dq_p,\mathcal L_dq\big)_{L^2(\varOmega^+)}=-\big( u_2,\Delta \mathcal L_dq\big)_{L^2(\varOmega^+)}.
\end{equation}
In the same manner, for the second term in the right hand side of \eqref{ghnjko}, there exists a function $ u_3$ in $H^1(\varOmega^+)$ such that:
\begin{equation}
\big(\chi\mathcal L_d q_p,\chi\mathcal L_d q\big)_{L^2(\varOmega^+)}=\big(\chi^2\mathcal L_d q_p,\mathcal L_d q\big)_{L^2(\varOmega^+)}=
-\big( u_3,\Delta \mathcal L_dq\big)_{L^2(\varOmega^+)}.
\end{equation}
\end{subequations}
Using the expressions \eqref{cvwqso} in \eqref{mulkjiopp} and \eqref{ghnjko}, we obtain eventually:
$$\llangle p,q\rrangle_{-\frac12,\frac12,d}=-\big(v_p,\Delta \mathcal L_dq\big)_{L^2(\mathbb R^2)}\forallt q\in\big(\mathscr A^{1/2}_{S,d}\big)^\perp,$$
where $v_p=\Pi_{\varOmega^+}\big(2 u_1+ u_2+ u_3-\chi\Delta\chi\mathcal L_d q_p\big)$ and $\Pi_{\varOmega^+}$ stands for 
the Bergman projection in $L^2(\varOmega^+)$. 
\par
It remains to construct $\mathsf q_p\in\mathscr A_S^{-1/2}$ as announced in the statement of the theorem. Let $\{P_1,P_2,P_3\}$ be a basis of $\mathscr A$ such that $(P_k,P_j)_{L^2(\varOmega^-)}=\delta_{j,k}$. For every $j=1,2,3$, 
let $\tilde{\mathsf q}_j\in\mathscr A_S^{-1/2}$ and  $\hat q_j\in\mathscr A_{S,d}^{1/2}$ 
be such that $\mathscr S_\varGamma\tilde{\mathsf q}_j|_{\varOmega^-}=\Delta\mathcal L_d\hat q_j|_{\varOmega^-}=P_j|_{\varOmega^-}$. 
If follows that 
$$\big(\mathsf S_\varGamma\tilde{\mathsf q}_j,\hat q_k\big)_{L^2(\varGamma)}=(P_j,P_k)_{L^2(\varOmega^-)}=\delta_{j,k}\forallt j,k=1,2,3.$$
This proves that we can always define $\mathsf q_p\in \mathscr A_S^{-1/2}$ such that:
$$\big(\mathsf S_\varGamma\mathsf q_p,q\big)_{L^2(\varGamma)}=-\llangle 
p,q\rrangle_{-\frac12,\frac12,d}-\big(v_p,\Delta \mathcal L_dq\big)_{L^2(\varOmega^+)}\forallt q\in \mathscr A^{1/2}_{S,d},$$
and completes the proof.
\end{proof}
 Let us address now the Neumann problems. Recall that:
 $$\widetilde{\mathcal H}_n^{-3/2}(\varGamma)=\big\{q\in {\mathcal H}_n^{-3/2}(\varGamma)\,:\, \llangle q,\mathbf 1_\varGamma 
 \rrangle_{-\frac32,\frac32,n}=0\}.$$
\begin{theorem}[Solvability of the interior Neumann problem]
\label{neuman_inter}
For every $q\in\widetilde{\mathcal H}^{-3/2}_n(\varGamma)$ 
 there exists $v_q$ in $\mathscr H^0(\varOmega^-)$ such that $\gamma_n^-v_q=q$. Moreover, the application $q\longmapsto v_q$ 
 is continuous from $\widetilde{\mathcal H}^{-3/2}_n(\varGamma)$ to $\mathscr H^0(\varOmega^-)$.
\end{theorem}
We omit the proof which is similar to that of Theorem~\ref{diri_interior}.
 \begin{theorem}[Solvability of the exterior Neumann problem]
 \label{ghbppl}
For every $q\in\widetilde{\mathcal H}^{-3/2}_n(\varGamma)$   there exists $v_q$ in $\mathscr H^0(\varOmega^+)$ and $\mathsf p_q\in\mathscr A^{1/2}_D$ such that $\gamma_n^+\big(v_q+\mathscr D_\varGamma \mathsf p_q\big)=q$.
\end{theorem}
The proof is roughly the same as the one of Theorem~\ref{diri_exterior} and rests on the following lemma:
\begin{lemma}
\label{lbnjiolkiol}
Let $\xi$ be a distribution in $H^{-1}(\varOmega^+)$  either compactly supported in $\varOmega^+$ or in $L^2(\varOmega^+)$ 
and compactly supported in $\overline{\varOmega^+}$ and define the constant 
$\alpha= |\varGamma|^{-1}\langle\xi,\mathbf 1_{\varOmega^+}\rangle_{H^{-1}(\varOmega^+),
H^1(\varOmega^+)}$. Then there exists  $\mathsf q\in \mathscr A_D^{- 1/2}$  such that the Dirichlet problem
\begin{equation}
\label{diri2}
-\Delta u=\xi\quad\text{ in }\varOmega^+\qquad\text{and}\qquad \gamma_n^+u=\alpha\mathbf 1_\varGamma+\mathsf q\quad\text{ on }\varGamma
,
\end{equation}
admits  a solution in $H^1(\varOmega^+)$. The solution is in $H^1(\varOmega^+,\Delta)$ if $\xi$ is in $L^2(\varOmega^+)$.
\end{lemma}
If $\xi$ is compactly supported in $\varOmega^+$, the solution $u$ is harmonic near the boundary $\varGamma$ and the 
Neumann trace is well defined. If $\xi$ is in $L^2(\varOmega^+)$, then $u$ is in $H^1({\varOmega^+},\Delta)$ and again the boundary condition makes sens.
\begin{proof}
The only notable difference with the proof of Lemma~\ref{lmmme} is that the space $W^1_0(\varOmega^+)$ must be replaced with 
the space:
$$W^1(\varOmega^+)=\Big\{u\in \mathscr D'(\varOmega^+)\,:\, \frac{u}{\rho\,{\rm lg}}\in L^2(\varOmega^+),\,\frac{\partial u}{\partial x_j}\in L^2(\varOmega^+),\,(j=1,2)\Big\},$$
provided with the scalar product:
$$(\nabla u_1,\nabla u_2)_{L^2(\varOmega^+,\mathbb R^2)}+\mu(\gamma_du_1)\mu(\gamma_du_2)\forallt u_1,u_2\in W^1(\varOmega^+),$$
whose corresponding norm is equivalent to the natural norm.
\end{proof}

\section{Transmission problems}
\label{SEC:trans}
We continue assuming that $\varGamma$ is a straight polygon. We are interested in the following transmission problems:
\par
\medskip
\noindent{\underline{\bf Problem~1}:} Let $p$ be in $\mathcal H_d^{-1/2}(\varGamma)$ and $q$ be in  $\mathcal H^{-3/2}(\varGamma)$. Find $u\in L^2_{\ell oc}(\mathbb R^2)$ such that, for some $a\in\mathbb R$:
\begin{subequations}
\begin{empheq}[left=\empheqlbrace]{align}
\Delta u &=0\quad\text{in }\varOmega^-\cup\varOmega^+\\
[\gamma_du]_\varGamma&=0,\\
\label{53c}\gamma_d u& =p \quad\text{or}\quad [\gamma_nu]_\varGamma=q,\\
u(x)&=a\ln|x|+\mathscr O(1/|x|)\quad\text{as}\quad|x|\longrightarrow+\infty.
\end{empheq}
\end{subequations}
\begin{theorem}
\label{nkmpq}
Problem~1 admits always a solution. Any solution $u$ is a single layer potential $\mathscr S_\varGamma^\dagger \bar q$ for some 
$\bar q\in \mathcal H^{-3/2}(\varGamma)$. This solution is unique if   condition \eqref{53c} is  $[\gamma_nu]_\varGamma=q$, in which case 
$\bar q=q$. If condition \eqref{53c} is $\gamma_d u =p$, the solution is not unique in general.
\end{theorem}
\begin{proof}
Let $u$ be a solution to Problem~1. Then, according to \eqref{asymp_sl}, there exists $\mathsf q\in\mathscr A_S^{-1/2}$ such that the function
$v=u-\mathscr S_\varGamma\mathsf q$ belongs to $\mathscr H^0(\mathbb R^2\setminus\varGamma)$. This functions 
satisfies furthermore $[\gamma_d v]_\varGamma=0$ which means, applying point 2 of Theorem~\ref{mpsdt}, that $v$ and hence also $u$ 
is a single layer potential.
\par
If condition  \eqref{53c} is  $[\gamma_nu]_\varGamma=q$, the function 
$u=\mathscr S_\varGamma^\dagger  q$ is indeed a solution 
of the transmission problem. To prove uniqueness, assume that $u$ is a solution to the problem with $q=0$. According to \eqref{asymp_sl},
there exists $\mathsf q\in \mathscr A^{-1/2}_S$ such that $v = u+\mathscr S_\varGamma \mathsf q$ is in $L^2(\mathbb R^2)$. This function is in 
$\mathscr H^0(\mathbb R^2\setminus\varGamma)$ and satisfies $[\gamma_d v]_\varGamma=0$ and $[\gamma_nv]_\varGamma=\mathsf q$.
The second point of Theorem~\ref{mpsdt} asserts that $v=\mathscr S_\varGamma\mathsf q$ whence we deduce with \eqref{asymp_sl} again 
that $\mathsf q=0$, and then $u=0$.
\par
Assume now that condition \eqref{53c} is $\gamma_d u=p$ for some given $p$ in $\mathcal H^{-1/2}_d(\varGamma)$.
Applying Theorem~\ref{diri_exterior}, there exists $v_p^+$ in $\mathscr H^0(\varOmega^+)$ and
$\mathsf q_p\in \mathscr A^{-1/2}_S$ such that $\gamma_d^+(v_p^++\mathscr S_\varGamma \mathsf q_p)=p$. On the other hand, 
Theorem~\ref{diri_interior} provides us with a function $v_p^-\in\mathscr H^0(\varOmega^-)$ such that $\gamma_d^-v_p^-=p
-\gamma_d^-\mathscr S_\varGamma \mathsf q_p$. Define now $v_p\in\mathscr H^0(\mathbb R^2\setminus\varGamma)$
 by setting $v_p|_{\varOmega^+}=v_p^+$ and $v_p|_{\varOmega^-}=v_p^-$.  Since $[\gamma_d v_p]_\varGamma=0$ we 
 are entitled to apply Theorem~\ref{mpsdt} which ensures us that $v_p=\mathscr S_\varGamma^\dagger [\gamma_n v_p]_\varGamma$.
It follows that $u=\mathscr S^\dagger_\varGamma \bar q$ with $\bar q=[\gamma_n v_p]_\varGamma+\mathsf q_p$.
\end{proof}
From this proof, we easily deduce:
\begin{cor}\label{nkmpq:cor}For every $u$ in $\mathscr H^0(\varOmega^-)$, there exists $q\in\mathcal H^{-3/2}(\varGamma)$ such that 
$\mathscr S_\varGamma^\dagger q|_{\varOmega^-}=u$. For every $u$ in $L^2_{\ell oc}(\overline{\varOmega^+})$, harmonic and such that 
$u(x)=a\ln|x|+\mathscr O(1/|x|)$ as $|x|\longrightarrow+\infty$ (for some $a\in\mathbb R$) there exists $q\in\mathcal H^{-3/2}(\varGamma)$ such that 
$\mathscr S_\varGamma^\dagger q|_{\varOmega^+}=u$.
\end{cor}
\noindent{\underline{\bf Problem~2}:} Let $p$ be in $\mathcal H^{-1/2}(\varGamma)$ and $q$ be in  $\widetilde{\mathcal H}^{-3/2}_n(\varGamma)$. 
Find $u\in L^2_{\ell oc}(\mathbb R^2)$ such that:
\begin{subequations}
\begin{empheq}[left=\empheqlbrace]{align}
\Delta u &=0\quad\text{in }\varOmega^-\cup\varOmega^+\\
[\gamma_nu]_\varGamma&=0,\\
\label{54c}
 [\gamma_du]_\varGamma&=p\quad\text{or}\quad\gamma_n u =q,\\
u(x)&=\mathscr O(1/|x|)\quad\text{as}\quad|x|\longrightarrow+\infty.
\end{empheq}
\end{subequations}
The proofs of the following Theorem and Corollary are omitted because they are
 similar to those of Theorem~\ref{nkmpq} and Corollary~\ref{nkmpq:cor}. Introducing the space:
 $$\widetilde{\mathcal H}^{-1/2}(\varGamma)=
\big\{p\in \mathcal H^{-1/2}(\varGamma)\,:\, \llangle p,\mathbf 1_\varGamma\rrangle_{-\frac12,\frac12}=0\big\},$$
they are stated as follows:
\begin{theorem}
\label{kkqpolop}
Problem~2 admits always a solution. Any solution $u$ is a double layer potential   $\mathscr D_\varGamma^\dagger \bar p$ for some 
$\bar p\in \mathcal H^{-1/2}(\varGamma)$. This solution is unique if   condition \eqref{54c} is  $[\gamma_d u]_\varGamma=p$, in which case 
$\bar p=p$. If condition \eqref{54c} is $\gamma_n u =q$, $\bar p$ can be chosen in $\widetilde{\mathcal H}^{-1/2}(\varGamma)$ 
and the solution is not unique in general.
\end{theorem}
\begin{cor}\label{kkqpolop:cor}For every $u$ in $\mathscr H^0(\varOmega^-)$, there exists $p\in\mathcal H^{-1/2}(\varGamma)$ such that 
$\mathscr D_\varGamma^\dagger p|_{\varOmega^-}=u$. For every $u$ in $L^2_{\ell oc}(\overline{\varOmega^+})$, harmonic and 
such that 
$u(x)=\mathscr O(1/|x|)$ as $|x|\longrightarrow+\infty$ there exists $p\in\widetilde{\mathcal H}^{-1/2}(\varGamma)$ such that 
$\mathscr D_\varGamma^\dagger p|_{\varOmega^+}=u$.
\end{cor}
%
\par
\noindent{\underline{\bf Problem~3}:} Let $p\in\mathcal H^{-1/2}_d(\varGamma)$ and $q\in\mathcal H^{-3/2}_n(\varGamma)$ be such that 
$p\in \mathcal H^{1/2}(\varGamma)$ or $q\in\mathcal H^{-3/2}(\varGamma)$. 
Find $u\in L^2_{\ell oc}(\mathbb R^2)$ such that, for some $a\in\mathbb R$:
\begin{subequations}
\label{system55}
\begin{empheq}[left=\empheqlbrace]{align}
\Delta u &=0\quad\text{in }\varOmega^-\cup\varOmega^+\\
\label{55b}
[\gamma_d u]_\varGamma&=p\quad\text{and}\quad[\gamma_nu]_\varGamma=q,\\
u(x)&=a\ln|x|+\mathscr O(1/|x|)\quad\text{as}\quad|x|\longrightarrow+\infty.
\end{empheq}
\end{subequations}
%
\begin{theorem}
Problem~3 admits a unique solution  given by $u=\mathscr S^\dagger_\varGamma q+
\mathscr D^\dagger_\varGamma p$.
\end{theorem}
\begin{proof}
According to the first point of Theorem~\ref{mpsdt}, $(p,q)\in  \mathcal H^{-1/2}(\varGamma)\times  \mathcal H^{-3/2}(\varGamma)$ and 
$u=\mathscr S^\dagger_\varGamma q+
\mathscr D^\dagger_\varGamma p$ solves System \eqref{system55}. The uniqueness is proved in the same way as in the proof of  Theorem~\ref{nkmpq}.
\end{proof}
\begin{prop}
On the contrary to what happens for functions in $\mathscr H^1(\mathbb R^2\setminus\varGamma)$ (see \eqref{ghaqpl}), there exist functions $u\in\mathscr H^0(\mathbb R^2\setminus\varGamma)$ that cannot be achieved as the sum of a single and a double 
layer potential.
\end{prop}
\begin{proof}
Let $u^-$ be in $\mathscr H(\varOmega^-)$ such that $\gamma_d^-u^-=p$ with $p\in\mathcal H_d^{-1/2}(\varGamma)$ but 
$p\notin \mathcal H^{-1/2}(\varGamma)$. Define $u$ in $\mathscr H^0(\mathbb R^2\setminus\varGamma)$ by setting $u|_{\varOmega^-}
=u^-$ and $u|_{\varOmega^+}=0$. Then $[\gamma_d u]_\varGamma=p\notin  \mathcal H^{-1/2}(\varGamma)$ and therefore 
$u$ cannot be the sum of a single and a double layer potential.
\end{proof}
\par
We end this section with the question of representing the harmonic functions defined in Theorems~\ref{diri_interior}, \ref{diri_exterior}, \ref{neuman_inter} and 
\ref{ghbppl} as layer potentials. We need to define first:
$$\widetilde{\mathcal H}^{-1/2}(\varGamma)=
\big\{p\in \mathcal H^{-1/2}(\varGamma)\,:\, \llangle p,\mathbf 1_\varGamma\rrangle_{-\frac12,\frac12}=0\big\}.$$
\begin{theorem}
The bounded operators $\gamma_d\circ\mathscr S^\dagger_\varGamma:\mathcal H^{-3/2}(\varGamma)\longrightarrow \mathcal H^{-1/2}_d(\varGamma)$, $\gamma_n^-\circ\mathscr S^\dagger_\varGamma:\mathcal H^{-3/2}(\varGamma)\longrightarrow \widetilde{\mathcal H}^{-3/2}_n(\varGamma)$ and $\gamma_n^+\circ\mathscr S^\dagger_\varGamma:\mathcal H^{-3/2}(\varGamma)\longrightarrow {\mathcal H}^{-3/2}_n(\varGamma)$
are surjective but not injective in general. The same conclusion applies for 
 $\gamma_n\circ\mathscr D^\dagger_\varGamma:\widetilde{\mathcal H}^{-1/2}(\varGamma)\longrightarrow \widetilde{\mathcal H}^{-3/2}_n(\varGamma)$  and 
 $\gamma_d^-\circ\mathscr D^\dagger_\varGamma:{\mathcal H}^{-1/2}(\varGamma)\longmapsto \mathcal H^{-1/2}_d(\varGamma)$.
 \end{theorem}
 \begin{proof}
 The subjectivity of $\gamma_d\circ\mathscr S^\dagger_\varGamma:\mathcal H^{-3/2}(\varGamma)\longrightarrow \mathcal H^{-1/2}_d(\varGamma)$ 
 results from Theorem~\ref{nkmpq}.
 \par
  Let now $q$ be given in $\widetilde{\mathcal H}^{-3/2}_n(\varGamma)$. According to Theorem~\ref{neuman_inter}, 
 there exists a function $v^-$ in $\mathscr H(\varOmega^-)$ whose normal trace $\gamma_n^-v^-$ is equal to $q$. 
 We apply next Theorem~\ref{diri_exterior} which asserts the existence of a function $v^+$ (the sum of a function in $\mathscr H^0(\varOmega^+)$ 
 and a single layer potential) such that $\gamma_d^+v^+=\gamma_d^-v^-$. The function $v$ defined by 
 $v|_{\varOmega^+}=v^+$ and $v|_{\varOmega^-}=v^-$. This function $v$ is a solution to the transmission problem~1 and 
 therefore, according to Theorem~\ref{nkmpq}, it is a single layer potential, what proves that 
 $\gamma_n^-\circ\mathscr S_\varGamma^\dagger$ is surjective. 
 \par
 Let us verify that the range of $\gamma_n^+\circ\mathscr S_\varGamma^\dagger$ is $\mathcal H^{-3/2}(\varGamma)$. Any $q$ 
 in $\mathcal H^{-3/2}(\varGamma)$  can be decomposed as $\bar q+\alpha\mathbf 1_\varGamma$
 with $\bar q\in\widetilde{\mathcal H}^{-3/2}(\varGamma)$ and 
 $\alpha=|\varGamma|^{-1}\langle q,\mathbf 1_\varGamma\rangle_{-\frac12,\frac12} $.
 According to Theorem~\ref{ghbppl}, there exists $v_{\bar q}\in \mathscr H^0(\varOmega^+)$ and $\mathsf p_{\bar q}$ in $\mathscr A_D^{1/2}$ 
 such that $\gamma_n^+(v_{\bar q}+\mathscr D^\dagger_\varGamma \mathsf p_{\bar q})=\bar q$. Denote by $p$ the external 
 one-sided Dirichlet trace $\gamma_d^+(v_{\bar q}+\mathscr D^\dagger_\varGamma \mathsf p_{\bar q})$ which belongs to 
 $\mathcal H^{-1/2}_d(\varGamma)$ and apply Theorem~\ref{diri_interior}:
 There exists $u^-\in \mathscr H^0(\varOmega^-)$ such that $\gamma_d^-u^-=p$. Define now a function $v$ by setting $v|_{\varOmega^+}=\big(v_{\bar q}+\mathscr D^\dagger_\varGamma \mathsf p_{\bar q}\big)|_{\varOmega^+}$   and $v|_{\varOmega^-}=u^-$. 
 Since $[\gamma_d v]_\varGamma=0$, we can apply Theorem~\ref{nkmpq} and conclude that $v$ is a single layer potential. 
 Denote by $\mathsf e_\varGamma$ the equilibrium density of $\varGamma$  i.e. the unique element in $H^{-1/2}(\varGamma)$ such that 
 $\gamma_d\circ\mathscr S_\varGamma \mathsf e_\varGamma$ is a constat function on $\varGamma$ normalized in such a way 
 that $\langle \mathsf e_\varGamma,\mathbf 1_\varGamma\rangle_{-\frac12,\frac12}=1$ (see \cite[page 263]{McLean:2000aa}). 
 The function $v+(\alpha/c_\varGamma)\mathscr S_\varGamma\mathsf e_\varGamma$ ($c_\varGamma$ is the constant value taken by $\mathscr S_\varGamma\mathsf e_\varGamma$ on $\varGamma$) is a preimage of $q$ by $\gamma_n^+$. 
 \par
The remaining two results are proved in the same way, so the proof is omitted.
 \end{proof}
 The last operator however deserves a special treatment:
 \begin{theorem}
 Let $p$ be given in $\mathcal H^{-1/2}_d(\varGamma)$. There exists a constant $c$ and $\bar p\in \widetilde{\mathcal H}^{-1/2}(\varGamma)$ such that $\gamma_d^+\circ\mathscr D^\dagger_\varGamma\bar p=p+c$.
 \end{theorem}
 \begin{proof}
 Let $p$ be given in $\mathcal H^{-1/2}_d(\varGamma)$. According to Theorem~\ref{nkmpq}, there exists $\bar q\in\mathcal H^{-3/2}(\varGamma)$ such 
 that $\gamma_d\circ \mathscr S_\varGamma^\dagger \bar q=p$. Next, denote by $\tilde q$ 
 the external one-sided Neumann trace $\gamma_n^+\circ \mathscr S^\dagger_\varGamma\bar q$ 
and let $\alpha\in\mathbb R$  be such that $\llangle \tilde q+\alpha \mathsf e_\varGamma,
 \mathbf 1_\varGamma\rrangle_{-\frac32,\frac32,n}=0$. It holds $\gamma_d\circ \mathscr S_\varGamma^\dagger (\bar q+\alpha\mathsf e_\varGamma)=p+\alpha
 c_\varGamma$ (recall that $c_\varGamma$ is the constant value taken by the function $\gamma_d\circ\mathscr S_\varGamma\mathsf e_\varGamma$) and
 $\gamma_n^+\circ \mathscr S_\varGamma^\dagger (\bar q+\alpha\mathsf e_\varGamma)=\tilde q+\alpha \mathsf e_\varGamma$ since 
 $\gamma_n^-\circ \mathscr S_\varGamma \mathsf e_\varGamma=0$. We define now a function $v$ be setting 
 $v|_{\varOmega^+}=\big(\mathscr S_\varGamma^\dagger (\bar q+\alpha\mathsf e_\varGamma)\big)|_{\varOmega^+}$ and $v|_{\varOmega^-}$ 
 is the solution, provided by Theorem~\ref{neuman_inter}, to the interior Neumann problem 
 with boundary data $\tilde q+\alpha\mathsf e_\varGamma$. The function $v$ is a solution to Problem~2 
 and therefore, according to Theorem~\ref{kkqpolop}, it is a double layer potential.
  \end{proof}
\appendix
\section{List of the main function spaces and operators}
\label{list_of_content}
\subsubsection*{Weigthed Sobolev spaces}
The space
\vspace{-5mm}

\begin{align*}
W^2(\mathbb R^2)&=\Big\{u\in \mathscr D'(\mathbb R^2)\,:\, \frac{u}{\rho^2\,{\rm lg}}\in L^2(\mathbb R^2),\, \frac{1}{\rho\,{\rm lg}}\frac{\partial u}{\partial x_j}\in L^2(\mathbb R^2)~\text{ and }~\frac{\partial^2 u}{\partial x_j\partial x_k}\in L^2(\mathbb R^2),\,\forall\,j,k=1,2\Big\},\\[-1mm]
\intertext{and its subspaces}
{W_d^2}(\mathbb R^2)&=\big\{u\in W^2(\mathbb R^2)\,:\, \gamma_du=0\big\},\\
{W_n^2}(\mathbb R^2)&=\big\{u\in W^2(\mathbb R^2)\,:\, \gamma_nu=0\big\},\\
\mathscr B_n(\mathbb R^2)&=\big({W^2_d}(\mathbb R^2)\cap {W^2_n}(\mathbb R^2)\big)^\perp\cap {W^2_n}(\mathbb R^2),\\
\mathscr B_d(\mathbb R^2)&=\big({W^2_d}(\mathbb R^2)\cap {W^2_n}(\mathbb R^2)\big)^\perp\cap {W^2_d}(\mathbb R^2),\\
\mathscr A&=\{(x_1,x_2)\longmapsto a+b_1x_1+b_2x_2\,:\, a,b_1,b_2\in\mathbb R\}\quad\text{(the affine functions)}
\end{align*}
are provided with either one of the scalar products:

\vspace{-5mm}
\begin{align*}
(\cdot,\cdot)_S&=(\Delta\, \cdot,\Delta\, \cdot)_{L^2(\mathbb R^2)}+
\sum_{j=1}^3\langle  \mathsf q_j,\gamma_d \,\cdot\rangle_{-\frac12,\frac12}\langle \mathsf q_j,\gamma_d \,\cdot\rangle_{-\frac12,\frac12}\\[-2mm]
(\cdot,\cdot)_D&=(\Delta\,\cdot,\Delta \,\cdot)_{L^2(\mathbb R^2)}+\sum_{j=1}^2( \mathsf p_j,\gamma_n \,\cdot)_{L^2(\varGamma)}( \mathsf p_j,\gamma_n \,\cdot)_{L^2(\varGamma)}
+\mu(\gamma_d \,\cdot)\mu(\gamma_d \,\cdot).
\end{align*}
\subsubsection*{Boundary spaces}
\begin{alignat*}{3}
{\mathcal H}^{3/2}(\varGamma)&=\gamma_d W^2(\mathbb R^2)\qquad&\text{and}&\qquad {\mathcal H}^{1/2}(\varGamma)=\gamma_n W^2(\mathbb R^2).\\
{\mathcal H}^{3/2}_n(\varGamma)&=\gamma_dW_n^2(\mathbb R^2)\qquad&\text{and}&\qquad {\mathcal H}^{1/2}_d(\varGamma)=\gamma_nW^2_d(\mathbb R^2),\\
\mathscr A_S^{1/2}&=\gamma_d \mathscr A\qquad&\text{and}&\qquad\mathscr A_S^{-1/2}=\mathsf S_\varGamma^{-1}\mathscr A_S^{1/2},\\
\mathscr A_D^{-1/2}&=\gamma_n \mathscr A\qquad&\text{and}&\qquad\mathscr A_D^{1/2}=\mathsf D_\varGamma^{-1}\mathscr A_D^{-1/2}.
\end{alignat*}
$$
\begin{array}{|l|l|l|}
\hline
\text{Space and dual space} & \text{Duality bracket}&\text{Scalar product}\\
\hline
H^{1/2}(\varGamma),\quad H^{-1/2}(\varGamma) & \langle\cdot,\cdot\rangle_{-\frac12,\frac12} & (\cdot,\cdot)_{\frac12}=
\langle\mathsf S_\varGamma^{-1}\cdot,\cdot\rangle_{-\frac12,\frac12}\\
\mathcal H^{1/2}(\varGamma),\quad \mathcal H^{-1/2}(\varGamma) & \llangle\cdot,\cdot\rrangle_{-\frac12,\frac12} & (\cdot,\cdot)^A_{\frac12}=
(\mathsf L_n^A\,\cdot,\mathsf L_n^A\,\cdot)_A,\quad A\in\{S,D\}\\
\mathcal H^{3/2}(\varGamma),\quad \mathcal H^{-3/2}(\varGamma) & \llangle\cdot,\cdot\rrangle_{-\frac32,\frac32} & (\cdot,\cdot)^A_{\frac32}=
(\mathsf L_d^A\,\cdot,\mathsf L_d^A\,\cdot)_A,\quad A\in\{S,D\}\\
\mathcal H^{1/2}_d(\varGamma),\quad \mathcal H^{-1/2}_d(\varGamma) & \llangle\cdot,\cdot\rrangle_{-\frac12,\frac12,d} & (\cdot,\cdot)_{\frac12,d}
=(\mathcal L_d\,\cdot,\mathcal L_d\,\cdot)_S\\
\mathcal H^{3/2}_n(\varGamma),\quad \mathcal H^{-3/2}_n(\varGamma) & \llangle\cdot,\cdot\rrangle_{-\frac32,\frac32,n} & (\cdot,\cdot)_{\frac32,n}
=(\mathcal L_n\,\cdot,\mathcal L_n\,\cdot)_D\\
\hline
\end{array}
$$
\begin{align*}
\widetilde{\mathcal H}^{-1/2}(\varGamma)&=
\big\{p\in \mathcal H^{-1/2}(\varGamma)\,:\, \llangle p,\mathbf 1_\varGamma\rrangle_{-\frac12,\frac12}=0\big\},\\
\widetilde{\mathcal H}_n^{-3/2}(\varGamma)&=\big\{q\in {\mathcal H}_n^{-3/2}(\varGamma)\,:\, \llangle q,\mathbf 1_\varGamma 
 \rrangle_{-\frac32,\frac32,n}=0\}.
 \end{align*}
%
\subsubsection*{Some isometric operators}
$A=S$ or $A=D$ in the definitions below:
\begin{align*}
\mathsf L_d^A:\big({\mathcal H}^{3/2}(\varGamma),\|\cdot\|_{\frac32}^{A}\big)&\longrightarrow \big({W^2_d}(\mathbb R^2)^\perp,\|\cdot\|_A\big)\\
p&\longmapsto\inf\big\{\|u\|_A\,:\,u\in W^2(\mathbb R^2),\,\gamma_du=p\big\},\\[2mm]
\mathsf L_n^A:\big({\mathcal H}^{1/2}(\varGamma),\|\cdot\|_{\frac12}^{A}\big)&\longrightarrow \big({W^2_n}(\mathbb R^2)^\perp,\|\cdot\|_A\big)
\\
q&\longmapsto\inf\big\{\|u\|_A\,:\,u\in W^2(\mathbb R^2),\,\gamma_nu=q\big\},\\[2mm]
{\mathcal L}_n:\big({\mathcal H}^{3/2}_n(\varGamma),\|\cdot\|_{\frac32,n}\big)&\longrightarrow 
\big(\mathscr B_n(\mathbb R^2),\|\cdot\|_D\big)\\
p&\longmapsto\inf\big\{\|u\|_S\,:\,u\in {W^2_n}(\mathbb R^2),\,\gamma_du=p\big\},\\[2mm]
\mathcal L_d:\big({\mathcal H}^{1/2}_d(\varGamma),\|\cdot\|_{\frac12,d}\big)&\longrightarrow 
\big(\mathscr B_d(\mathbb R^2),\|\cdot\|_S\big)\\
q&\longmapsto\inf\big\{\|u\|_S\,:\,u\in {W^2_d}(\mathbb R^2),\,\gamma_nu=q\big\}.
\end{align*}
%
\subsubsection*{Continuous and dense inclusions}
\begin{gather*}
{\mathcal H}^{3/2}_n(\varGamma)\subset {\mathcal H}^{3/2}(\varGamma)\subset H^{1/2}(\varGamma)\subset L^2(\varGamma)
\subset H^{-1/2}(\varGamma)\subset {\mathcal H}^{-3/2}(\varGamma)\subset {\mathcal H}^{-3/2}_n(\varGamma),\\
{\mathcal H}^{1/2}_d(\varGamma)\subset {\mathcal H}^{1/2}(\varGamma)\subset L^2(\varGamma)
\subset {\mathcal H}^{-1/2}(\varGamma)\subset {\mathcal H}^{-1/2}_d(\varGamma).
\end{gather*}
%
\subsubsection*{Further isometric operators}
$$
\fct{\mathsf T_{\! d}:{\mathcal H}^{3/2}(\varGamma)}{\mathcal H^{-3/2}(\varGamma)}
{p}{(p,\cdot)_{\frac32}^S}
\qquad\text{and}\qquad
\fct{\mathsf T_{\! n}:{\mathcal H}^{1/2}(\varGamma)}{\mathcal H^{-1/2}(\varGamma)}
{q}{(q,\cdot)_{\frac12}^D,}
$$
$$
\fct{\mathcal T_{\! d}:{\mathcal H}^{1/2}_d(\varGamma)}{\mathcal H^{-1/2}_d(\varGamma)}
{q}{(q,\cdot)_{\frac12,d},}
\qquad\text{and}\qquad
\fct{\mathcal T_{\! n}:{\mathcal H}^{3/2}_n(\varGamma)}{\mathcal H^{-3/2}_n(\varGamma)}
{p}{(p,\cdot)_{\frac32,n}.}
$$
We use $L^2(\varGamma)$ as pivot space, so none of these operators reduce to the identity.



\end{document}

%% file: simple_coucge_2.tex
\begingroup%
  \makeatletter%
  \providecommand\color[2][]{%
    \errmessage{(Inkscape) Color is used for the text in Inkscape, but the package 'color.sty' is not loaded}%
    \renewcommand\color[2][]{}%
  }%
  \providecommand\transparent[1]{%
    \errmessage{(Inkscape) Transparency is used (non-zero) for the text in Inkscape, but the package 'transparent.sty' is not loaded}%
    \renewcommand\transparent[1]{}%
  }%
  \providecommand\rotatebox[2]{#2}%
  \newcommand*\fsize{\dimexpr\f@size pt\relax}%
  \newcommand*\lineheight[1]{\fontsize{\fsize}{#1\fsize}\selectfont}%
  \ifx\svgwidth\undefined%
    \setlength{\unitlength}{277.30465386bp}%
    \ifx\svgscale\undefined%
      \relax%
    \else%
      \setlength{\unitlength}{\unitlength * \real{\svgscale}}%
    \fi%
  \else%
    \setlength{\unitlength}{\svgwidth}%
  \fi%
  \global\let\svgwidth\undefined%
  \global\let\svgscale\undefined%
  \makeatother%
  \begin{picture}(1,0.44312557)%
    \lineheight{1}%
    \setlength\tabcolsep{0pt}%
    \put(0,0){\includegraphics[width=\unitlength,page=1]{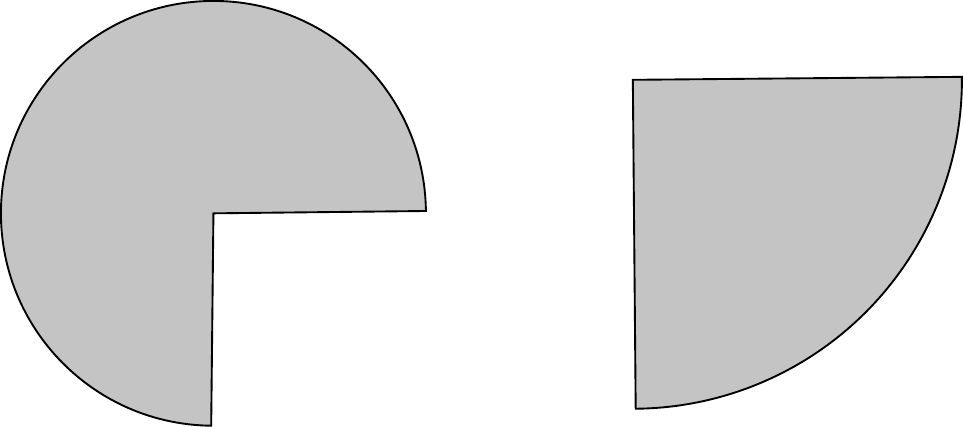}}%
    \put(0.08325249,0.24751321){\color[rgb]{0,0,0}\makebox(0,0)[lt]{\lineheight{1.25}\smash{\begin{tabular}[t]{l}$\varOmega^-$\end{tabular}}}}%
    \put(0.73278961,0.20302499){\color[rgb]{0,0,0}\makebox(0,0)[lt]{\lineheight{1.25}\smash{\begin{tabular}[t]{l}$\varOmega^-$\end{tabular}}}}%
    \put(0.43868426,0.37560535){\color[rgb]{0,0,0}\makebox(0,0)[lt]{\lineheight{1.25}\smash{\begin{tabular}[t]{l}$\varOmega^+$\end{tabular}}}}%
  \end{picture}%
\endgroup%